\documentclass{amsart}
\usepackage{amsmath}
\usepackage{hyperref}
\usepackage{graphicx}
\usepackage{color}
\usepackage[T1]{fontenc}
\usepackage[utf8]{inputenc}
 \usepackage{mathtools}
\usepackage{amsfonts}
\usepackage{changepage}
\usepackage{todonotes}
\usepackage{xcolor}
\usepackage[percent]{overpic}
\usepackage[rawfloats]{floatrow}

\newcommand{\A}{\mathcal A}

\theoremstyle{definition}
\newtheorem{defn}{Definition}[section]

\newtheorem{example}[defn]{Example}
\newtheorem{remark}[defn]{Remark}

\theoremstyle{plain}
\newtheorem{theorem}[defn]{Theorem}
\newtheorem{prop}[defn]{Proposition}

\author{Allaoua Boughrira}
\email{aboughrira@student.ccc.edu}

\author{Hellen Colman}
\email{hcolman@ccc.edu}

\address{Department of Mathematics, Wright College,
 4300 N.\ Narragansett Avenue, Chicago, IL 60634 USA}

\thanks{H.C. is supported in part by the Simons Foundation.}

\title[A Motion Planning Algorithm in a lollipop graph]{A Motion Planning Algorithm in a lollipop graph}


\begin{document}

\begin{abstract} This paper is concerned with problems relevant to motion planning in robotics. Configuration spaces are of practical relevance in designing safe control schemes for robots moving on a track. The topological complexity of a configuration space is an integer which can be thought of as the minimum number of continuous instructions required to describe how to move robots between any initial configuration to any final one without collisions. We calculate this number for various examples of robots moving in different tracks represented by graphs. We present and implement an explicit algorithm for two robots to move autonomously and without collisions on a lollipop track.

\end{abstract}
\maketitle

\section{Introduction}
We study the navigation problem for two robots to move autonomously and without collisions on a lollipop-shaped track. 

This paper surveys results concerning the motion planning problem for robots moving on graphs and includes a new explicit algorithm for the case of two robots on a lollipop graph. This algorithm is optimal in the sense that the motion planning is performed with the minimal number of instabilities.

First we introduce the configuration space of distinct robots on a graph. Configuration spaces in mathematics were introduced in the sixties by Fadell and Neuwirth \cite{Fadell} and first used in robotics in the eighties \cite{Lozano,Latombe}. More recently, configuration spaces of robots moving on graphs have been studied by Farber \cite{farber2004collision} and Ghrist \cite{ghrist2001configuration} amongst others. 

The configuration space for two robots moving on a graph is the space of all feasible combined positions of the robots. The motion planning problem deals with assigning paths between initial and final configurations. We are interested in algorithms that assign outputs (paths) to inputs (initial and final configurations) in a continuous way. These algorithms are rare in real world situations since most algorithms will have discontinuities.  Farber introduced the notion of {\em topological complexity} of configuration spaces which measures the discontinuities in algorithms for robot navigation \cite{Farber2003}.  The topological complexity of a space is invariant under homotopy.

Our approach consists on presenting an explicit construction of the configuration space for the case of two robots moving along a lollipop track and then building a deformation retract of it that we call the skeleton. We calculate the topological complexity of the skeleton of the configuration space and exhibit an algorithm with the minimal number of instructions. Finally, we translate back our instructions to the physical space where the robots move. We provide a detailed algorithm with concrete instructions for two robots to move on a lollipop track between any initial and final positions. We implement this algorithm as well and show a simulation for several cases. 

We organize the paper in the following manner.  In Section 2 we introduce some basic notions and notations. Section 3 introduces the basic  ideas on continuous motion planning. Section 4 presents the topological complexity and its basic properties.  Section 5 provides a detailed study for the case of two robots moving on a circle track. We construct the configuration space, calculate its topological complexity and present an algorithm for this case that will set the basis for our main example.
Section 6 concerns the characterization of our main example: two robots moving on a lollipop track. We construct the configuration space as before and give an explicit algorithm for this case. We also implement and show a simulation for this case.

This paper is the result of an undergraduate research project at Wilbur Wright College supervised by Professor Hellen Colman.

\section{Preliminaries}
In this section we briefly review some basic definitions and establish notations that we will use in this paper. From now on, the interval, circle and disk will be denoted as $I, S^1$ and $D$ respectively:
\begin{center}
\begin{equation}
\nonumber
\begin{aligned}
I\: &= [0,1] \subset \mathbb{R} \\
S^1 &= \{(x,y) \in \mathbb{R}^2\ | \ x^2 + y^2 = 1\}\\
D\: &= \{(x,y) \in \mathbb{R}^2\ | \ x^2 + y^2 < 1\}
\end{aligned}
\end{equation}
\end{center}

\subsection{Product Topological Space}
Given two topological spaces $X$ and $Y$, consider the cartesian product: 
$$ X \times Y =\{(x,y)\ |\ x\in X, \ y \in Y\}.$$

The product space is the set $X \times Y$ endowed with the product topology.

\begin{example}
Let $A$ and $B$ be the following two subspaces of $\mathbb{R}^2$ in polar coordinates:
$$A = \{(r,\theta)\ |\ \theta \in [0,2\pi],\ r\ =\ \frac{\pi}{6}+|\sin(\theta+\frac{\pi}{6})^5|\}$$
$$B = \{(r,\theta)\ |\ \theta \in [0,\pi],\ r\ =\  3\cos(\theta)\}$$
\end{example}

\begin{figure}[h]
\includegraphics[width=0.4235\textwidth,keepaspectratio]{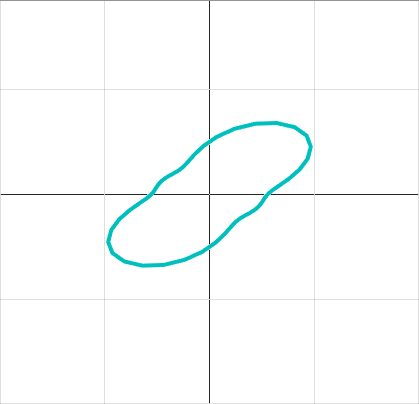}
\includegraphics[width=0.5100\textwidth,keepaspectratio]{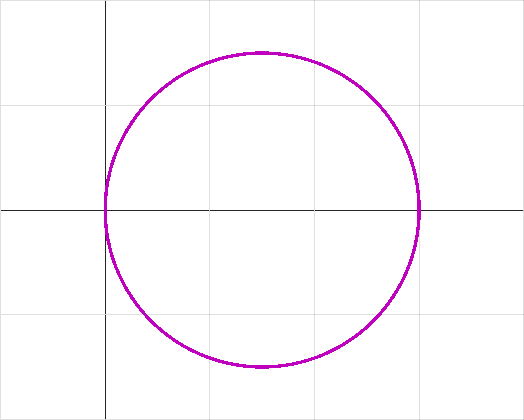}
\caption{Subspaces $A$ and $B$}
\end{figure}

The product space shown in figure \ref{productSpace} is the set $A \times B = \{(x,y) \ |\ x \in A,\ y \in B \}$ endowed with the product topology.

\begin{figure}[h]
\centering
\includegraphics[width=0.41\textwidth,keepaspectratio]{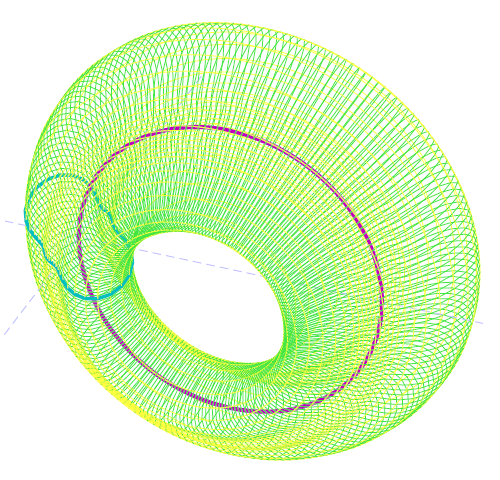}
\caption{Product space $A \times B$} 
\label{productSpace}
\end{figure}

\subsection{Homeomorphism and Homotopy}
A continuous function $f:X \to Y$ is a \emph{homeomorphism} if and only if $f^{-1}$ is also continuous and $f$ is a bijection.
 
Let $f$ and $g$ be two continuous maps between spaces $X$ and $Y$. Let $H:X \times I \to Y$ be a continuous map such that for each $x \in X$, $H(x,0) = f(x)$ and $H(x,1) =g(x)$, then $f$ and $g$ are \emph{homotopic} and $H$ is a \emph{homotopy} between $f$ and $g$. We denote $f \simeq g$.
Two spaces $X$ and $Y$ have the same \emph{type of homotopy} if there exist two functions $f:X \to Y$ and $g:Y \to X$ such that $f \circ g \simeq id_Y$ and $g \circ f \simeq id_X$.

Informally, we can say that one space has the same homotopy type as another if we can \emph{deform} one space into the other by compression, stretching and without \emph{tearing} the space nor \emph{gluing} any two parts of it. For instance, the space $A \times B$ in figure \ref{productSpace} has the same homotopy type as the \emph{torus} $T=S^1 \times S^1$ in figure \ref{Torus}. 

\begin{figure}[h]
\centering
\includegraphics[width=0.41\textwidth,keepaspectratio]{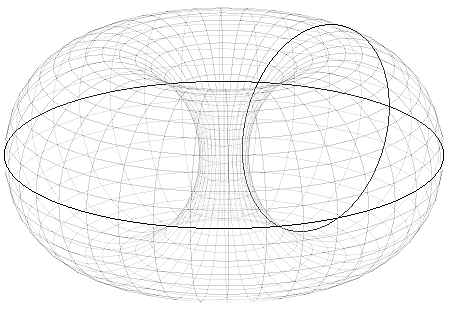}
\caption{Torus $T = S^1 \times S^1$}
\label{Torus}
\end{figure}

If two spaces are homeomorphic then they have the same type of homotopy.  The homotopy however does not imply homeomorphism.

A topological space $X$ is \emph{contractible} if $X$ has the same type of homotopy as a point. In other words, there exists a \emph{deformation} as described before from the space $X$ into a point. For instance, a disk $D$ is contractible whereas a circle $S^1$ is not.

\section{Motion Planning Algorithms}
In order to study the problem of planning the movement of robots in a certain space, we need to introduce certain notions and an associated space to aid with the statement as well as the solution of the problem.

\subsection{Physical and Configuration Space} 
The \emph{physical space} $\Gamma$ is the track where the robots move. The combined positions of all robots at any given moment defines a \emph{state} of the system. The space of all possible states is the \emph{configuration space}. Each state in the configuration space is a point that represents a unique configuration of the robots' locations in the physical space. The physical space is a special case of a configuration space where there is only one robot.
 
The \emph{configuration space} $X$ of $n$ robots moving on the physical space $\Gamma$ without collisions is:
$$X = C^n(\Gamma) = \underbrace{\Gamma \times \Gamma \times \ldots \times \Gamma}_{n \text{ times}}\ -\ \Delta_n = \Gamma^n - \Delta_n$$ where $\Delta_n$ is the \emph{diagonal}. The diagonal represents the set of pairs where the robots are co-located and is defined as:
 
$$\Delta_n = \{(x_1,x_2,\ldots, x_n) \in \Gamma^n\ |\ \exists j\neq i, \  x_i = x_j\}.$$
For all spaces $\Gamma$, we have that $C^1(\Gamma)=\Gamma$.
\subsection{Motion Planning} 
A path $\alpha$ in $X$ is a continuous function $\alpha:I\to X$, where $I$ is the interval [0,1]. The \emph{path space} $PX$ is the set of all paths in the space $X$,

$$PX = \{\alpha\ |\ \alpha\ \mbox{is a path in $X$}\}.$$

The evaluation function $ev$ is a map that takes in a path $\alpha$ in $X$ and returns the \emph{initial} and \emph{final} points of the path.

$$ev: PX \to X \times X$$
$$\alpha \mapsto (\alpha(0), \alpha(1))$$

A \emph{section} $s$ of the evaluation function is a function that takes in a pair of points in the space $X$ and gives out a \text{path} between them. That is, 
 
$$s\ : X\times X \to PX$$
$$(a,b) \mapsto \alpha_{a,b}$$
where $\alpha_{a,b}$ is an \emph{instruction} to move the robots from the starting point $a$ to the ending point $b$ following that specific path. A \emph{Motion Planning Algorithm} (MPA) is such a section.

\subsection{Continuity of Motion Planning}
Motion planning algorithms may or may not be \emph{continuous}. In a continuous motion planning algorithm, the small changes on the path are continuously dependent on the small changes that occur at the initial or final positions. In other words, the instruction to go from $a$ to $b$ will depend \emph{continuously} on the points $a$ and $b$.
 
The continuity of the MPA implies the stability of the robot behavior. That is, the \emph{errors} due to the imprecisions or uncertainties on the initial and final positions would still result in a \emph{nearby} path. 

We are interested in motion planning algorithms that assign instructions in a \emph{continuous way}. Unfortunately, we'll see later that these are very rare.

Let us consider for instance, one robot moving on a circle, and we will try to construct a section that would define a MPA. This robot can move between its initial and final position in different ways: clockwise, counter clockwise, following the shortest path, following the longest path, etc. 
 
Set the instruction to be ``follow the shortest path'', and the section would be defined as:

$$s:S^1\times S^1 \to PS^1$$
$$(a,b) \mapsto \alpha_{shortest}$$

We observe that this section is continuous in all pairs of positions except at the antipodal pairs, and so, this section is not continuous. 

Figure \ref{DiscontinueSh} shows the initial (filled triangle icon) and final position (empty triangle icon) of a robot. We see that a small variation in the final position results in a huge variation in the output path.

\begin{figure}[h] 
\centering
\includegraphics[width=0.8\textwidth,keepaspectratio]{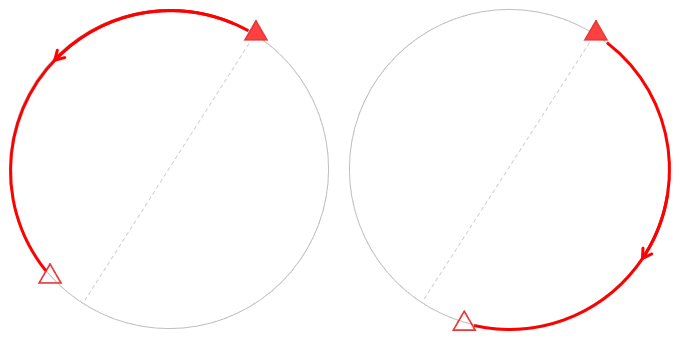}
\caption{Discontinuity at the antipodal pairs}
\label{DiscontinueSh}
\end{figure}

Now, let us try a different instruction, ``go counterclockwise'' for instance. Here also, there exists a state where the section is not continuous; that is the state $(a,a)$ for any $a$ in $S^1$. In other words where the initial and final positions coincide, the section presents a discontinuity (see figure \ref{DiscontinueCC}).

\begin{figure}[h] 
\centering
\includegraphics[width=0.8\textwidth,keepaspectratio]{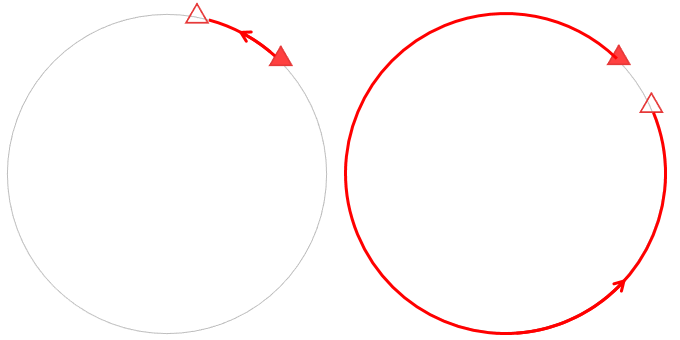}
 \caption{Discontinuity at configuration $(a,a)$}
\label{DiscontinueCC}
\end{figure}

For both instructions, the section presented discontinuities. In fact, we will see in the next that section that it is impossible to find a continuous section for this particular example.

\section{Topological Complexity}
Farber proved that a continuous MPA on a configuration space exists only if the space is contractible.

\begin{theorem}\cite{Farber2003} A globally defined continuous motion planning algorithm in $X$ exists if and only if the configuration space $X$ is contractible.
\end{theorem}

Since in the previous example the robot was moving on a circle $S^1$ which is also the configuration space in this case, and since a circle $S^1$ is not contractible, we have that there exists no continuous MPA for one robot moving on a circle.

We can however decompose the cartesian product of the configuration space into subdomains, with each having a continuous instruction.

 \begin{defn}\cite{Farber2003} The \emph{topological complexity} $TC(X)$ of a connected  space $X$ is defined as the \emph{minimum} integer  $k$ such that the cartesian product $X \times X$ can be covered by $k$ open subsets $U_1, U_2, ..., U_k$, such that for any $i=1,2,...,k$ there exists a \emph{continuous} section $s_i: U_i \to PX$ with $ev \circ s_i = i: U_i\hookrightarrow PX$.
\end{defn}

The sets $U_1, U_2, ..., U_k$ will be the called the \emph{domains of continuity}. If no such number $k$ exists, the $TC(X)$ will be set to be $\infty$. We observe that if the space $X$ is contractible, then $TC(X) = 1$.

\begin{example}
One robot moving on an interval $I$. In this case, we have that the configuration space is 
$X=C^1(I)=I$ and 
$TC(I) = 1$ with domain of continuity  $U =I \times I$.
Since $I$ is contractible, there exists a continuous MPA where the instruction would be for example ``go in straight line towards the final position''.
\end{example}

\begin{figure}[h]
\centering
\includegraphics[width=0.4\textwidth,keepaspectratio]{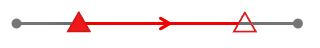}
\caption{One robot moving on I}
\end{figure}

\begin{example}\label{algoS1}
One robot moving on a circle $S^1$. Since there is just one robot, the configuration space is 
$X=C^1(S^1) = S^1$.
\end{example}

Since $S^1$ is not contractible, $TC(S^1)>1$. We have seen previously two examples of MPAs for one robot on $S^1$ where both had discontinuities. Now we can combine the two subsets where each instruction is continuous so that the union would cover the cartesian product $S^1 \times S^1$. 

Then, the domains of continuity will be:

$$U = \{(a,b) \in S^1 \times S^1\ |\ a\ \mbox{is not antipodal to}\ b\}\:\:\mbox{and}$$ 
$$V = \{(a,b) \in S^1 \times S^1\ |\ a\ \neq b\}.$$

Therefore, $TC(X) = 2$ since we exhibit an explicit covering by domains of continuity with exactly two sets.

From the MPA standpoint, this decomposition into $k$ subspaces can be interpreted as having a set of $k$ continuous instructions, where each operates in its corresponding domain of continuity. 
For this example, the resulting two instructions in the MPA would be as follows: If the initial and final positions are not antipodal then go following the shortest path. Otherwise, go counterclockwise.

\begin{figure}[h]
\centering
\includegraphics[width=0.75\textwidth,keepaspectratio]{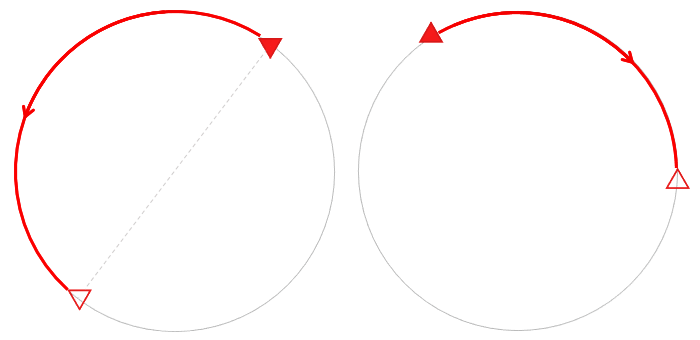}
\caption{Two Instructions for one robot moving on $S^1$}
\end{figure}

Farber also proved that the topological complexity $TC(X)$ of a space $X$ is \emph{invariant} under homotopy.
\begin{theorem}\cite{Farber2003}\label{homotopyInv} If $X \simeq Y$ then $TC(X) = TC(Y)$.
\end{theorem}
 
Computing the topological complexity of a configuration space is an essential step to derive the motion planning algorithm in the corresponding physical space since it provides the minimum number of instructions of the MPA. The examples seen so far have relatively simple configuration spaces, but as we increase the number of robots and the physical space gets more complex, so does the configuration space. Finding the discontinuities directly while in the configuration space is not trivial, and the above theorem is crucial to reduce the overall complexity by using the homotopy not only to calculate the topological complexity of a simpler space, but also to construct the actual algorithm.

\section{Two robots moving on a graph}
We consider the case of two \emph{distinct} robots $A$ and $B$ moving on a graph $\Gamma$. The problem to solve is finding a \emph{continuous} and \emph{collision-free} motion planning algorithm with the \emph{least} number of instructions needed to move both robots A and B from their initial positions to their final positions.

\subsection{Two robots moving on an interval $I$}
Let $A$ and $B$ be two robots moving on $\Gamma = I$. 

\begin{figure}[h]
\centering
\includegraphics[width=0.33\textwidth,keepaspectratio]{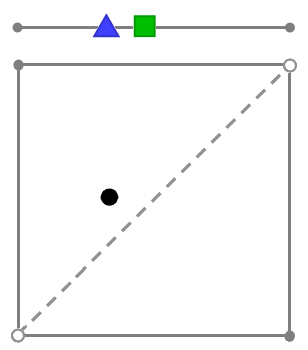}
\caption{Physical and configuration spaces for two robots in $I$}
\label{2RobotsOnI}
\end{figure}

The configuration space in this case is $X=C^2(I) = I \times I - \Delta$. The physical space and the configuration space are shown in figure \ref{2RobotsOnI}. Each pair of positions of the two robots in the physical space defines a state in the configuration space. Figure \ref{2RobotsOnI} shows a concrete pair of positions in the physical space and its corresponding state in the configuration space.

\begin{figure}[h]
\centering
\includegraphics[width=0.33\textwidth,keepaspectratio]{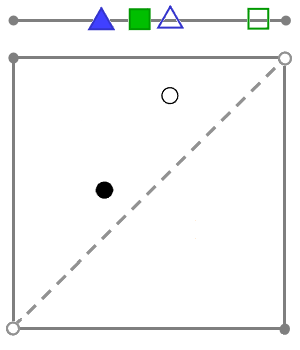}
\caption{Initial and final states are in the same connected component.}
\label{2RobotsOnI_a}
\end{figure}

We observe that the configuration space in this case is \emph{not connected} because of the removal of the diagonal. Even though there are some states where there is a path between the initial and final positions of the two robots (see figure \ref{2RobotsOnI_a}), the problem of finding a path between any initial and final configuration has no solution (figure \ref{2RobotsOnI_b}). In other words, there is no MPA for this case.

This impossibility to connect two states in the configuration space by a path has its counterpart in the physical space: two robots cannot swap positions within the interval.

\begin{figure}[H]
\centering
\includegraphics[width=0.33\textwidth,keepaspectratio]{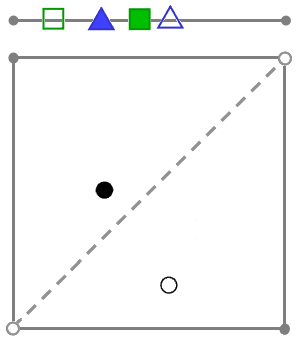}
\caption{Initial and final states are in different connected components.}
\label{2RobotsOnI_b}
\end{figure}

\subsection{Two robots moving on a circle $S^1$}
Now, let's consider two robots $A$ and $B$ moving on $\Gamma = S^1$. The configuration space in this case is given by $X=C^2(S^1) = S^1 \times S^1 - \Delta \ =\ T - \Delta.$ This space is homeomorphic to a cylinder $X=C^2(S^1) = S^1 \times I$.

\begin{figure}[h]
\centering
\includegraphics[width=0.75\textwidth,keepaspectratio]{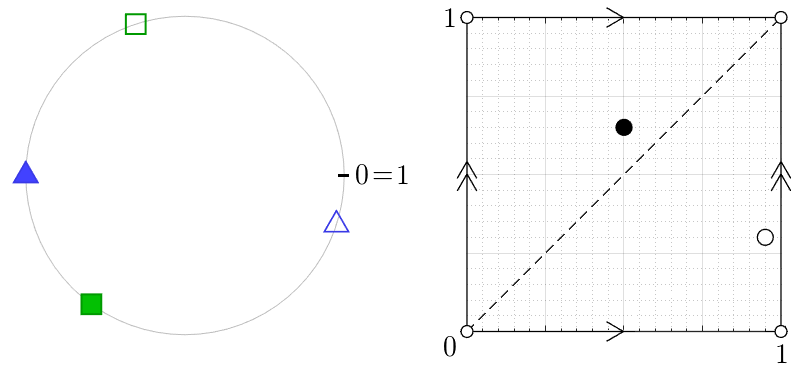}
\caption{Physical and configuration spaces for two robots on $S^1$}
\label{2RobotsOnS1}
\end{figure}	
A very useful way to represent this space would be using the \emph{flat torus} representation as shown in figure \ref{2RobotsOnS1}. We also shown in this figure the initial and final positions of two robots in the physical space and their corresponding states in the configuration space.

By working in the configuration space, the complexity of the problem is reduced from trying to find two paths to move two distinct robots from their initial positions to their respective final positions into finding one \emph{combined} path to move the initial state of the two robots to the final state in the configuration space.

\subsubsection{Topological Complexity of the Configuration Space} 
As we have seen before, the configuration space is $X=C^2(S^1)=S^1 \times I$. Since the configuration space is homotopy equivalent to a circle, we have that $TC(X) = TC(S^1 \times I) = TC(S^1) = 2$ by theorem \ref{homotopyInv}.
 
This gives us the minimum number of instructions needed for a continuous MPA for two robots moving on a circle $S^1$. Knowing how to transition back and forth between the physical space and the configuration space, our objective is first to find two instructions that define the algorithm to move from the initial state to the final state in the the configuration space, and then translate them back to the physical space. 

We will take advantage of the construction of an explicit homotopy between the configuration space and the circle to build our algorithm based on the known algorithm for the circle given in the example \ref{algoS1}.

\subsubsection{MPA for two robots moving on $S^1$}
Recall that the configuration space $X$ is path-connected and is homotopy equivalent to a circle. For the purpose of the homotopy, we will choose a special circle to deform the configuration space into: the \emph{antipodal circle}.

\begin{defn}
The \emph{antipodal circle} $\A\subset C^2(S^1)$ is the set of pairs of points on the circle such that they are antipodal to each other. We write:
$$\A = \{(a,b) \in S^1 \times S^1\:|\:a\:\mbox{is antipodal to}\: b\}$$
\end{defn}
\begin{figure}[h]
\centering
\includegraphics[width=1\textwidth,keepaspectratio]{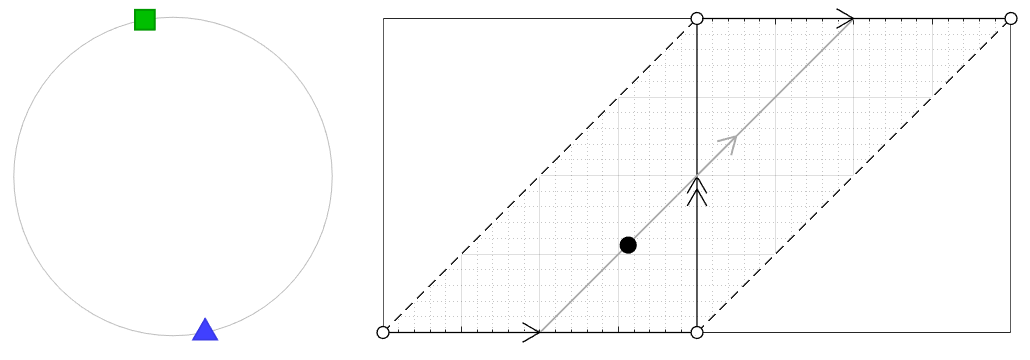}
\caption{The antipodal circle $\A$ and its counterclockwise orientation}
\label{AP}
\end{figure}	

The antipodal circle lays halfway of the diagonal $\Delta = \{(a,b)\in S^1 \times S^1\ |\ a=b\}$, and thus, it splits the configuration space into two equal and symmetric subspaces. In the flat torus representation (figure \ref{AP}), the antipodal circle $\A$ will be the set of points in $X$ lying on the line $b=a-\frac{1}{2}$.

Consider the homotopy $H:X \times I \to X$ given by the projection into $\A$ following horizontal lines.
Recall that each state $x=(A,B)$ represents the collective positions of robots $A$ and $B$ in the circle.

Let ${x_i}$ and $x_f$ be the initial and final states of the robots. We call the \emph{initial antipodal} state $x_i^\prime\:=\:H_1(x_i)$ the image of the initial state in the antipodal circle $\A$ by the homotopy. Likewise, we call the \emph{final antipodal} state $x_f^\prime\:=\:H_1(x_f)$ the projection per homotopy of the final state in  $\A$ (see figure \ref{AlgoInX}). 

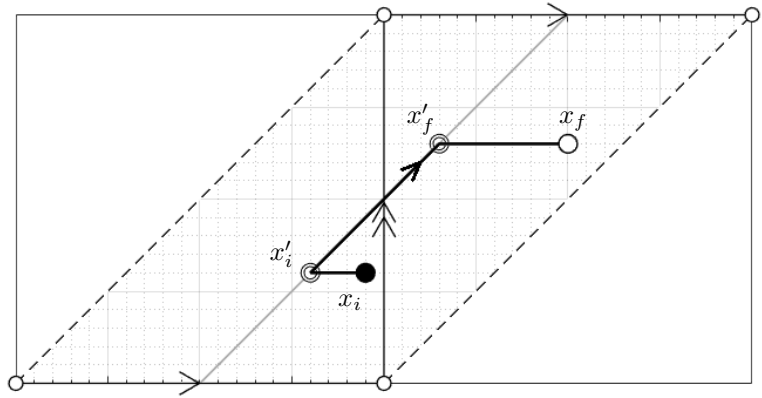
\begin{figure}[H]
\centering
\begin{overpic}
[width=0.8\textwidth,keepaspectratio]{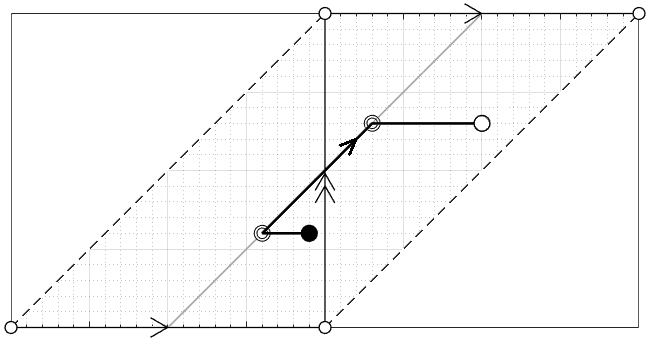}
\put(73,36){$x_f$}
\put(53.0,36){$x_f^\prime$}
\put(35,18){$x_i^\prime$}
\put(44,12){$x_i$}
\end{overpic}
\caption{Initial and final antipodal states}
\label{AlgoInX}
\end{figure}

The idea of our algorithm will be to move the initial and final states $x_i$ and $x_f$ to their antipodal states $x_i^\prime$ and $x_f^\prime$ respectively, and then apply in the circle $\A$ the known algorithm (example \ref{algoS1}) to move from $x_i^\prime$ to $x_f^\prime$ (see figure \ref{AlgoInX2}). 

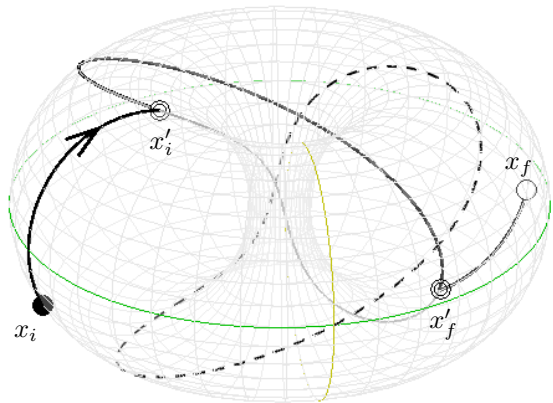
\begin{figure}[H]
\centering
\begin{overpic}
[width=0.6\textwidth,keepaspectratio]{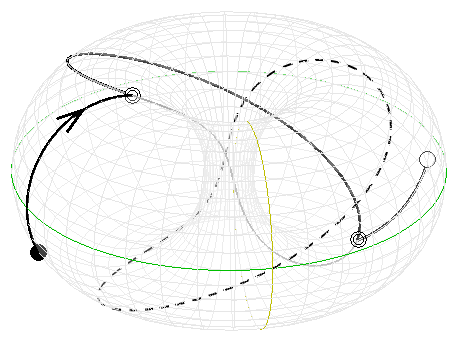}
\put(89.2,43){$x_f$}
\put(76,15){$x_f^\prime$}
\put(27,46.5){$x_i^\prime$}
\put(3.25,13.5){$x_i$}
\end{overpic}
\caption{The steps of the MPA in $X$}
\label{AlgoInX2}
\end{figure}	

We will also define certain distinguished positions in the configuration space that will correspond to the action of interchaging robots in the physical space.

\begin{defn}
Two states $x_i=(A_i,B_i)$ and $x_f=(A_f,B_f)$ are {\em swapped states} if $A_i=B_f$ and $A_f=B_i$.
\label{swapedDefinition}
\end{defn}

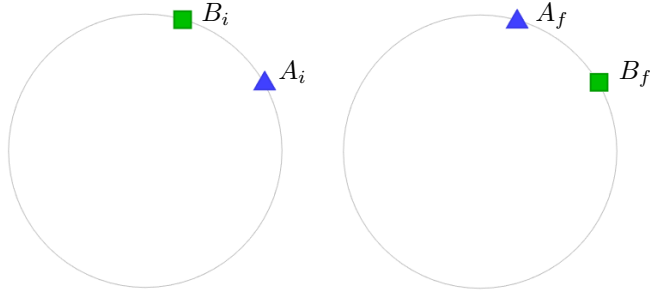
\begin{figure}[H]
\centering
\begin{overpic}
[width=0.675\textwidth,keepaspectratio]{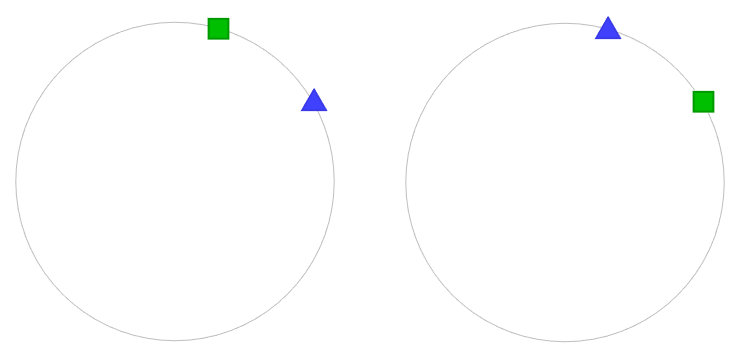}
\put(44,35){$A_i$}
\put(32,44){$B_i$}
\put(84,44){$A_f$}
\put(97,35){$B_f$}
\end{overpic}
\caption{Swapped states in the physical space}
\label{swapped states}
\end{figure}	

Swapped states correspond to points in the configuration space that are symmetric with respect to the diagonal (see figure \ref{swapConf}).

\begin{figure}[H]
\centering
\begin{overpic}
[width=0.675\textwidth,keepaspectratio]{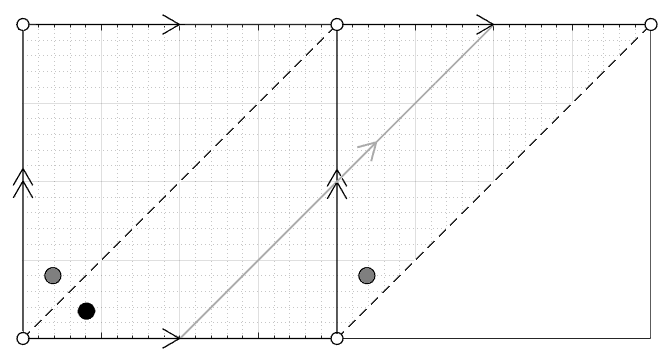}
\put(15.5,8.5){$x_i$}
\put(7.5,15.5){$x_f$}
\put(53.5,15.5){$x_f$}
\end{overpic}
\caption{Swapped states in the configuration space}
\label{swapConf}
\end{figure}
\begin{remark}
Observe that the swapped states that are in the antipodal circle $\A$ correspond exactly to the antipodal points in that circle. 
\end{remark}

The following is the description of our algorithm in the configuration space $X$.

\begin{enumerate}
\item Preliminary step: Move initial state $x_i$  to its antipodal state $x_i^\prime$ on $\A$ following the path  $H_t(x_i)$.
\item Main step: While on $\A$, if the two antipodal states $x_i^\prime$ and $x_f^\prime$ are swapped states, then move the initial antipodal state $x_i^\prime$ counterclockwise towards the final antipodal state $x_f^\prime$. Otherwise, move the initial antipodal state $x_i^\prime$ following the shortest path on $\A$ towards the final antipodal state $x_f^\prime$.

\item Final step: Move the final antipodal state $x_f^\prime$ back to the final state $x_f$ following the reverse path to $H_t(x_f)$.
\end{enumerate}

When transitioning back to the physical space, we obtain the following MPA to move two robots $A$ and $B$ in a circle:

\begin{enumerate}
\item Preliminary step: 
Move robot $A$ away from robot $B$ until robot $A$ reaches the antipodal position of the robot $B$. Robot $B$ remains stationary.\\

\item Main step:
If the final position of the robot $B$ is antipodal to its initial position, then move both robots in counterclockwise direction until robot $B$ reaches its final destination. Otherwise, move both robots in the same direction following the shortest path for robot $B$ to its final position, until robot $B$ reaches its final position.\\
 
\item Final step:
Move robot $A$ following the shortest path to its final position, until it reaches its final position. Robot $B$ remains stationary.
\end{enumerate}

Note that the preliminary and final steps are common steps in both instructions, the main step is where we deal with the discontinuity of the MPA.

The domains of continuity will be determined by the regions described before.
Let $V_1=\{((a,b),(a',b'))\in X\times X |\; b \mbox{ is antipodal to } b'\}$ and $V_2$ the set of all other pairs of points in the configuration space. Consider $U_1$ a small neighborhood of $V_1$ and $U_2=V_2$. Then $\{U_1, U_2\}$ is a covering of $X\times X $ by domains of continuity.

\subsection{Running the proposed MPAs}
\begin{example} Initial and final positions of the second robot are not antipodals.
 
\begin{figure}[H]
\centering
\includegraphics[width=1\textwidth,keepaspectratio]{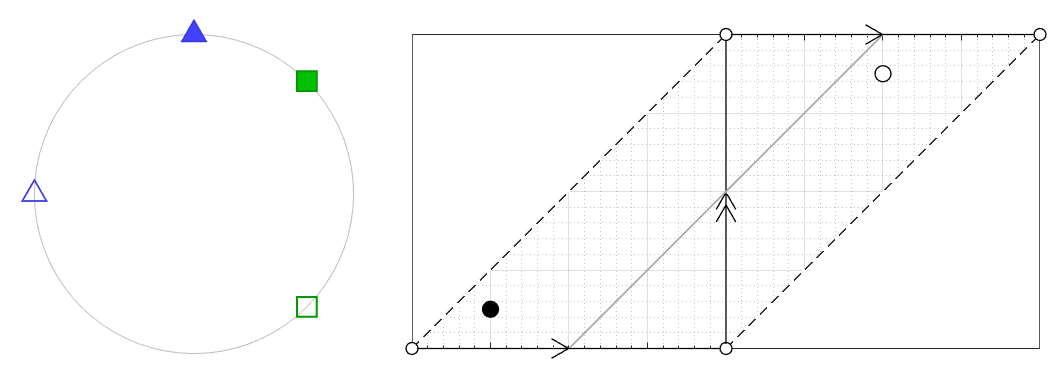}
\caption{Initial setup}
\end{figure}	
 
\begin{figure}[H]
\centering
\includegraphics[width=1\textwidth,keepaspectratio]{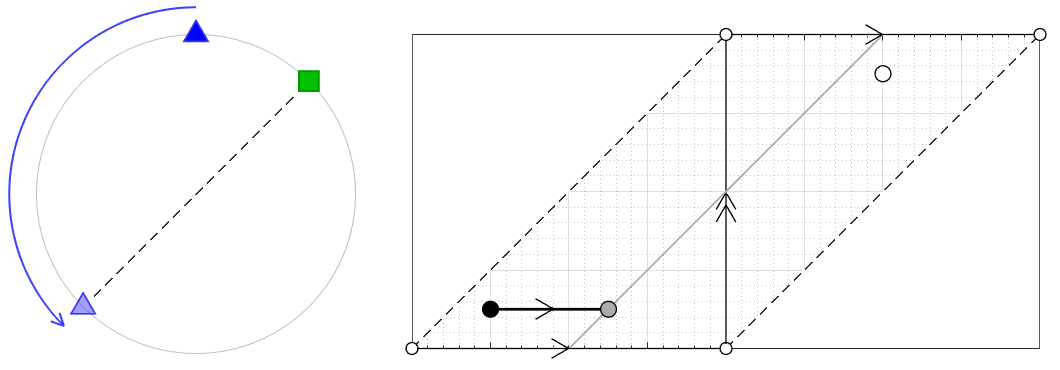} 
\caption{Preliminary step}
\end{figure}	
 
\begin{figure}[H]
\centering
\includegraphics[width=1\textwidth,keepaspectratio]{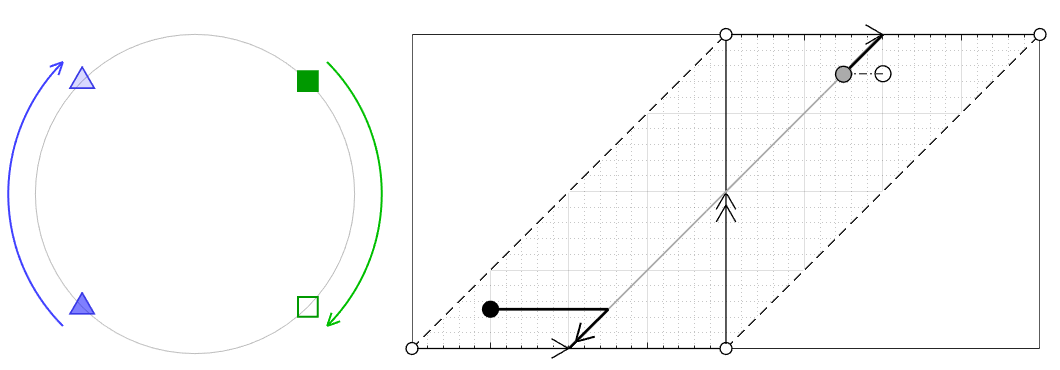}
\caption{Main step: "Move following the shortest path"}
\end{figure}	
 
\begin{figure}[H]
\centering
\includegraphics[width=1\textwidth,keepaspectratio]{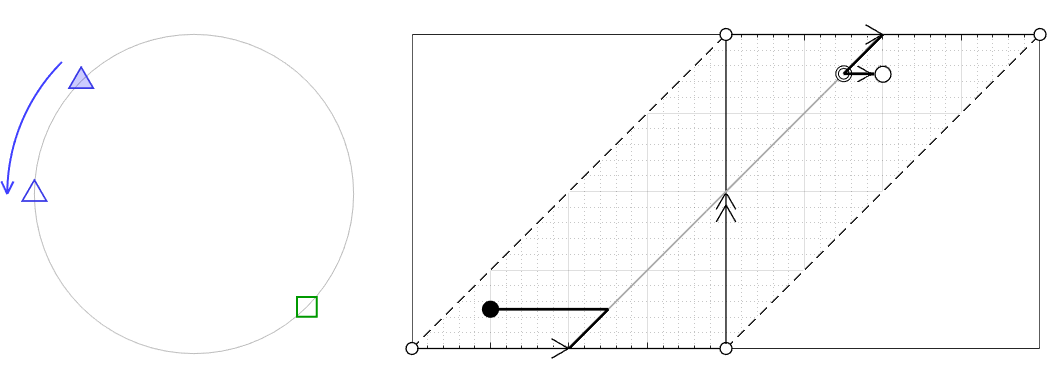}
\caption{Final step}
\end{figure}	
\end{example}

\begin{example} Initial and final positions of the second robot are antipodals.
 
\begin{figure}[H]
\centering
\includegraphics[width=1\textwidth,keepaspectratio]{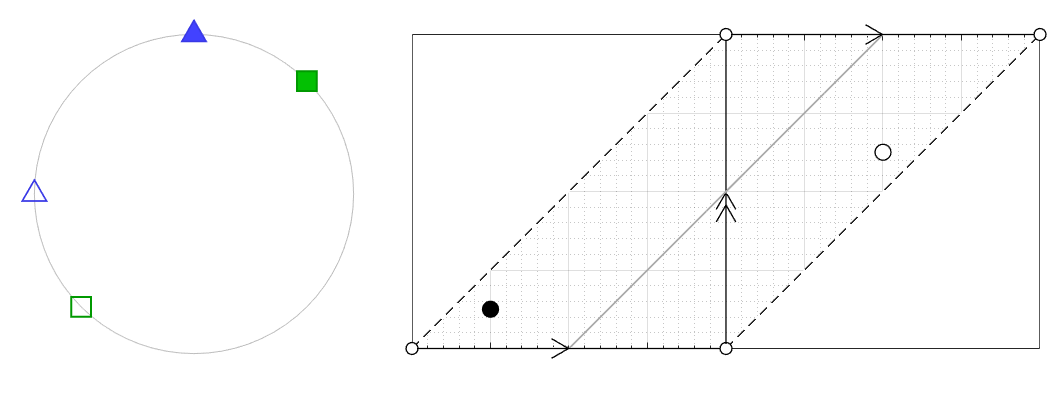}
\caption{Initial setup}
\end{figure}	
 
\begin{figure}[H]
\centering
\includegraphics[width=1\textwidth,keepaspectratio]{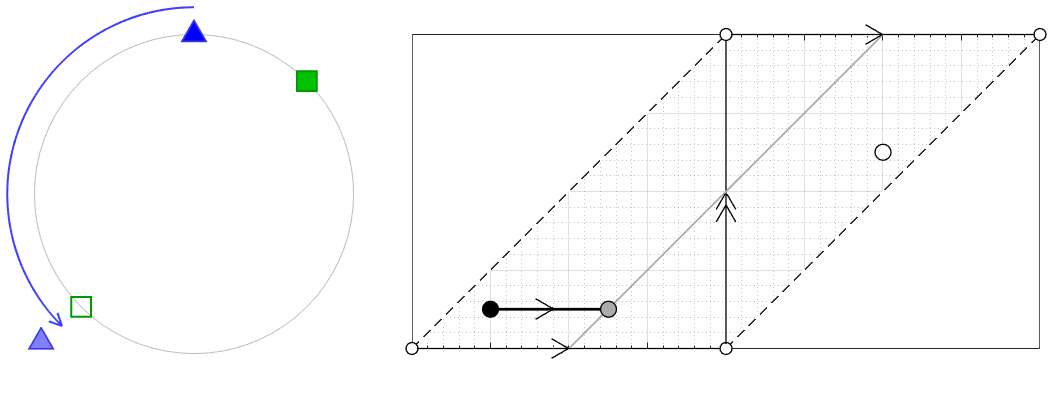} 
\caption{Preliminary step}
\end{figure}	
 
\begin{figure}[H]
\centering
\includegraphics[width=1\textwidth,keepaspectratio]{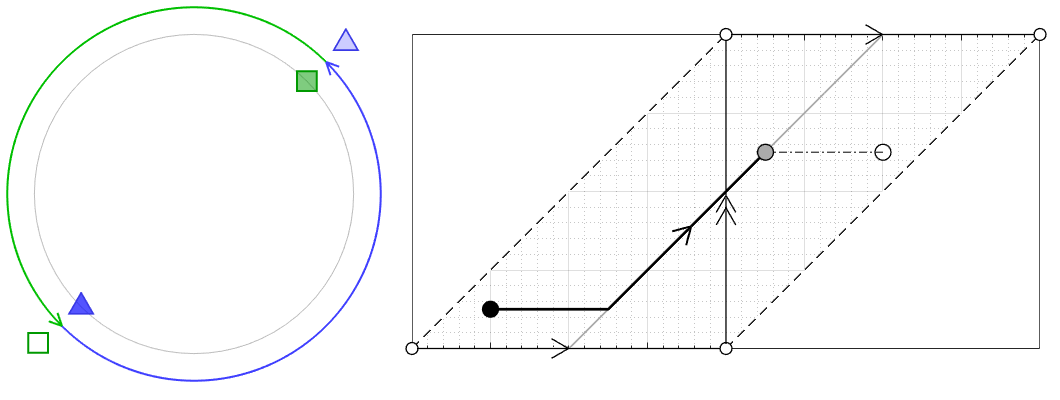}
\caption{Main step: "Move both robots in counterclockwise direction"}
\end{figure}	
 
\begin{figure}[H]
\centering
\includegraphics[width=1\textwidth,keepaspectratio]{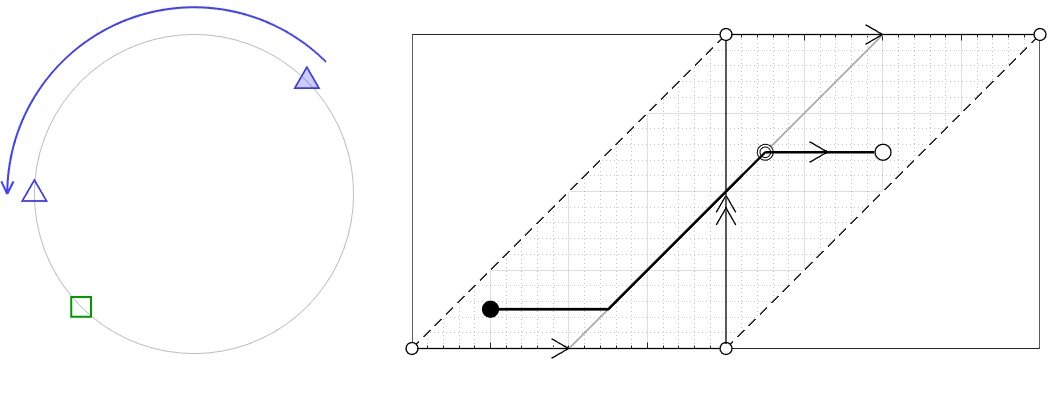}
\caption{Final step}
\end{figure}

\end{example}

In the next section, we'll expand this construction of the algorithm to derive a MPA for two robots moving on a lollipop-shape track.

\section{Main example: Lollipop graph $\Gamma=L$}
We focus our efforts now on the construction of a MPA for the case of two robots moving on a track consisting of a circle with an interval attached. 

\begin{figure}[H]
\centering
\begin{overpic}
[width=0.5\textwidth,keepaspectratio]{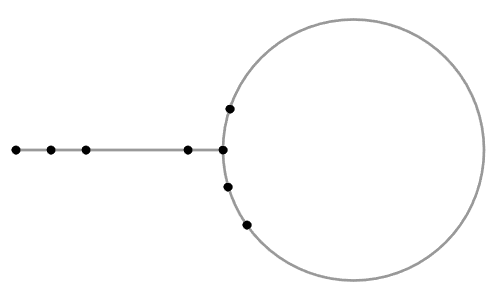}
\put(1.8,32){$1$}
\put(8.8,32){$2$}
\put(15.8,32){$3$}
\put(30,32){$n-g$}
\put(47,28.5){$1$}
\put(49,21){$2$}
\put(53,14){$3$}
\put(49,37){$g$}
\end{overpic}
\caption{Generalized lollipop graph $C_{n,g}$}
\label{GeneralL}
\end{figure}
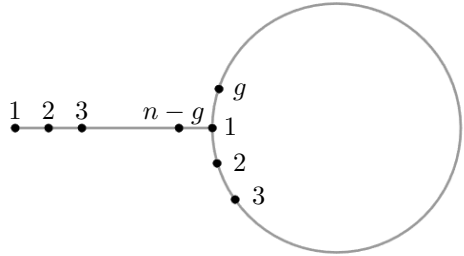

We consider the class $G_{n,g}$ of connected graphs on $n$ vertices with fixed girth $g$ of which the \emph{generalized lollipop graphs}, $C_{n,g}$, will be the ones obtained by appending a single $g$ cycle to a pendant vertex of a path on $n-g$ vertices \cite{fallat2002}. See figure \ref{GeneralL}.

\begin{figure}[H]
\centering
\begin{overpic}
[width=0.5\textwidth,keepaspectratio]{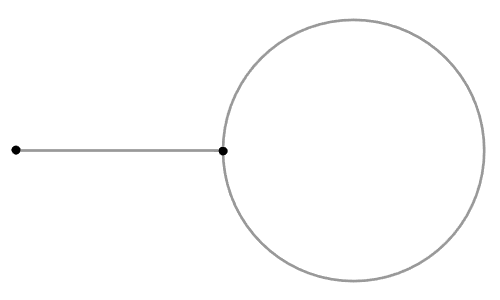}
\put(1.8,32){$1$}
\put(41,31.5){$2$}
\end{overpic}
\caption{Lollipop graph $L=C_{2,1}$}
\label{LollipopGraph}
\end{figure}

\begin{defn}
A \emph{lollipop} graph $L$ will be a generalized lollipop graph with two vertices and girth 1, $L = C_{2,1}$. See figure \ref{LollipopGraph}.
\end{defn}

\subsection{Configuration space}
The configuration space $X$ is given by: 
$$X=C^2(L) = L \times L - \Delta$$ 
where $L$ is a lollipop graph.

Figure \ref{configII} depicts $X$ and its flat representation. Recall that in the configuration space, each state represents the combined positions of the two robots in $L$, where the $x$-axis represents the positions of the first robot $A$, and the $y$-axis represents the positions of the second robot $B$ as in figure \ref{flatII}.

\begin{figure}[h]
\centering
\includegraphics[width=1\textwidth,keepaspectratio]{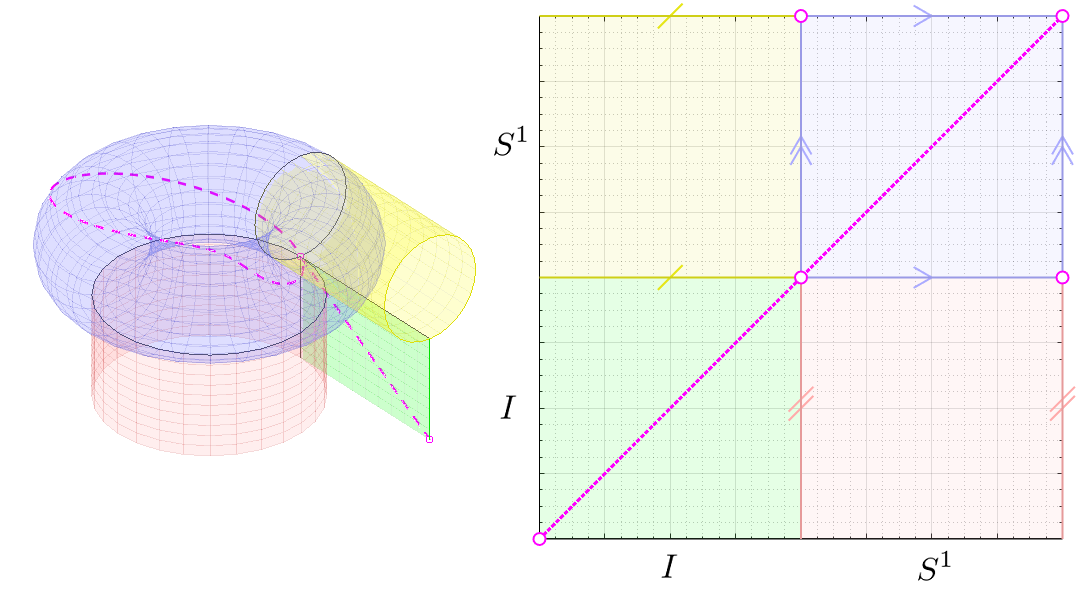}
\caption{Configuration space $X$ and its flat representation}
\label{configII}
\end{figure}

\begin{figure}[h]
\centering
\includegraphics[width=0.88\textwidth,keepaspectratio]{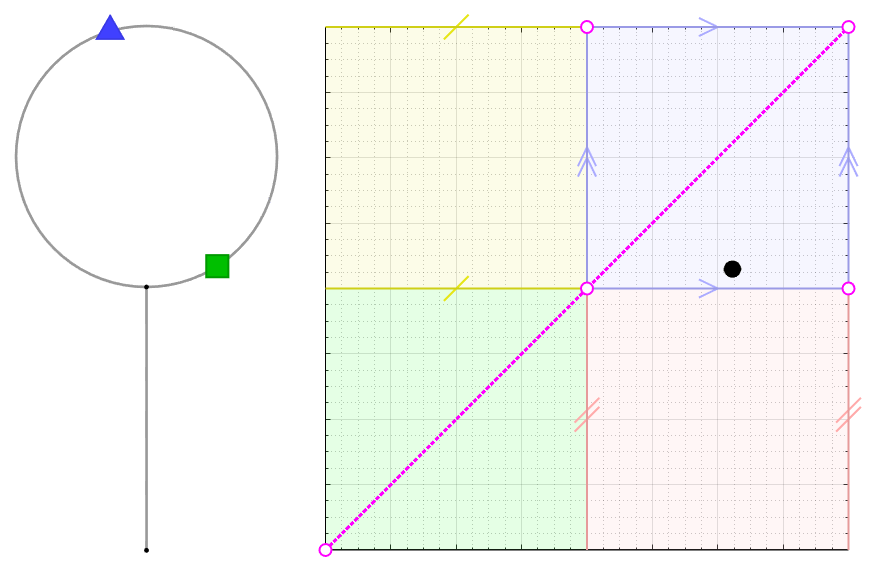}
\caption{Physical positions and their state in configuration space}
\label{flatII}
\end{figure} 

\subsubsection{Homotopy deformation} Ghrist proved in \cite{ghrist2002} that the configuration space of any graph with $k$ vertices of degree greater than two, deformation retracts to a subcomplex of dimension at most $k$ . We have then that our configuration space is homotopy equivalent to a 1-dimensional space. Our motion planning algorithm will be derived while working in the configuration space, as we explained in the previous examples. The execution of the algorithm, however, will run in the physical space. We will chose a 1-dimensional space in the configuration space that has a clear meaning when projected back into the physical space. Following the same approach as before, this 1-dimensional space, that we will call the skeleton $S$, will lay halfway from the diagonal representing the collision line if the two robots are both in the circle or in the interval. This graph will be completed with the addition of other segments as in figure \ref{LollipopFlat}.

\begin{figure}[H]
\centering
\includegraphics[width=0.88\textwidth,keepaspectratio]{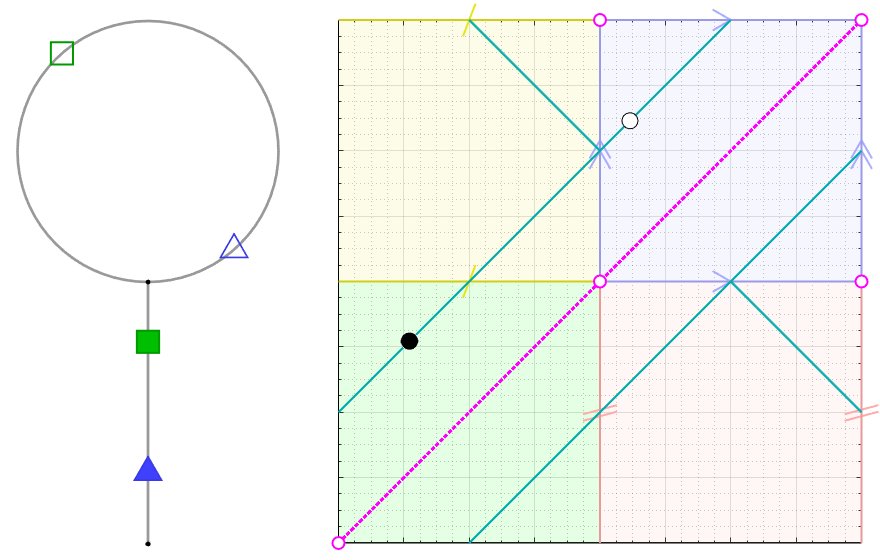}
\caption{Two positions in the skeleton $S$}
\label{LollipopFlat}
\end{figure}

The homotopy $H$ deforming the configuration space $X$ into the skeleton $S$ is shown in figure \ref{homotopyII}. This homotopy will be useful to transit in and out of the skeleton in the configuration space $X$.

\begin{figure}[H]
\centering
\includegraphics[width=0.6\textwidth,keepaspectratio]{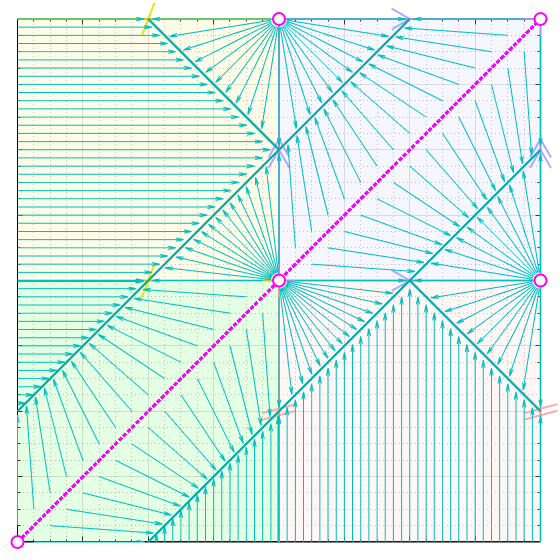}
\caption{Traces of the homotopy deformation $H$}
\label{homotopyII}
\end{figure}

The skeleton $S$ is homeomorphic to a chain of three circles  with two added intervals as shown in figure \ref{skelII}. By contracting the intervals and the lower part of each circle, we can see that the skeleton is homotopy equivalent to a wedge of three circles $S^1 \vee S^1 \vee S^1$ as shown in figure \ref{3wedged}.

\begin{figure}[H]
\centering
\includegraphics[width=0.75\textwidth,keepaspectratio]{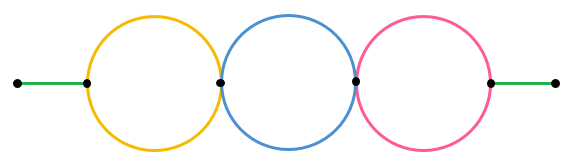}
\caption{The skeleton $S$}
\label{skelII}
\end{figure}

\begin{figure}[h]
\centering
\includegraphics[width=0.33\textwidth,keepaspectratio]{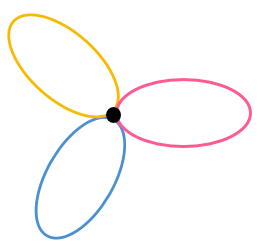}
\caption{Wedge of three circles}
\label{3wedged}
\end{figure}

\subsection{Topological Complexity}
In order to exhibit a set of instructions for the motion planning algorithm, we will consider the number of these continuous instructions given by the topological complexity  of the configuration space, $TC(X)$. Farber calculated the topological complexity of all graphs based on their first Betti number. 

In topological graph theory, the first Betti number of a graph $G$ with $n$ vertices, $m$ edges and $k$ connected components is given by $b_1(G) = m-n+k$.
\begin{prop}\cite{farber2004instabilities}\label{betti}
Let $G$ be a graph, then
\begin{equation}\nonumber
TC(G)=\left\{ 
	\begin{array}{cl}
		1 & \textrm{if }b_1(G) = 0\\
		2 & \textrm{if }b_1(G) = 1\\
		3 & \textrm{if }b_1(G) > 1
	\end{array}\right.
\end{equation}
\end{prop}

\begin{theorem}
The topological complexity of the configuration space of two robots moving on a lollipop graph is three, $TC(C^2(L))=3.$
\end{theorem}

\begin{proof}
We have shown that $C^2(L) \simeq S^1\vee S^1\vee S^1$. By proposition \ref{homotopyInv}, we know that $TC(C^2(L))= TC(S^1\vee S^1\vee S^1)$. 
Since the first Betti number of a wedge of three circles is $b_1=3-1+1=3$, we have that $TC(S^1\vee S^1\vee S^1)=3$ by proposition \ref{betti}. Therefore, $TC(C^2(L))=3$.
\end{proof}
We conclude that for two robots moving on a lollipop graph, any MPA would require at least three continuous instructions.

\subsection{Motion Planning Algorithm}
We proved previously that $TC(X)=3$ if $X=C^2(L)$, which defines the minimum number of continuous sections in $X \times X$. Recall that although the set of instructions of the algorithm will be derived in the configuration space, its execution will be run in the physical space.

If the state corresponding to the two robots' positions is on the skeleton $S$, we know that in the physical space the robots are one half unit apart from each other. Recall that the distance between two robots in $L$ is the minimal length of the paths in $L$ joining them. See figure \ref{shortestD} for examples of distances between robots.

\begin{figure}[H]
\centering
\includegraphics[width=0.97\textwidth,keepaspectratio]{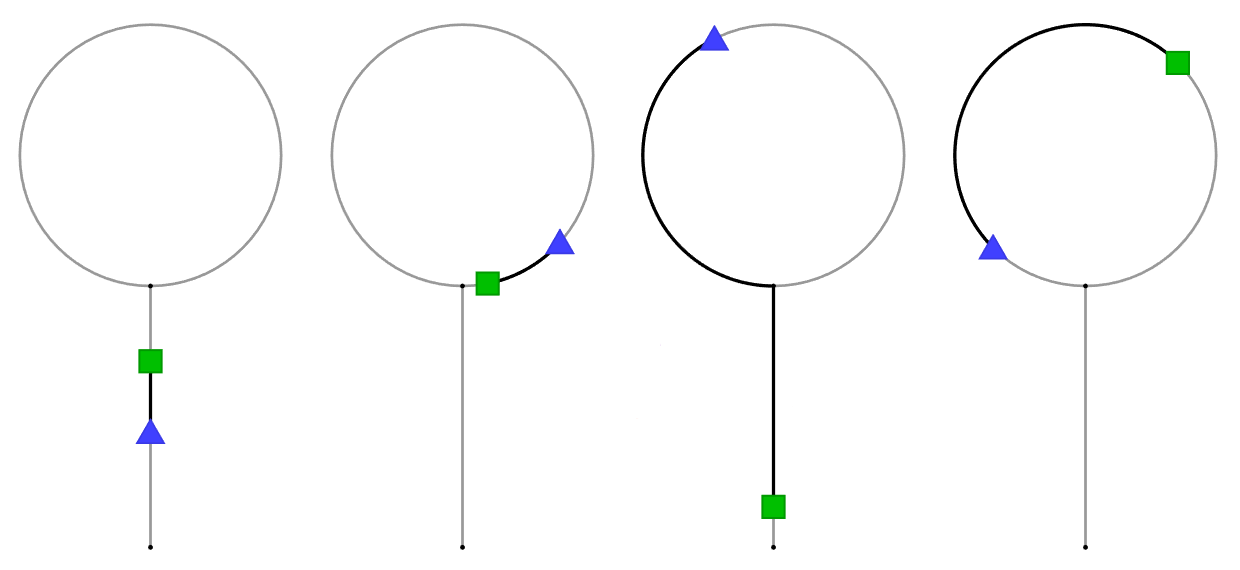}
\caption{Distance between two robots}
\label{shortestD}
\end{figure}

\begin{defn}
Two robots are in \emph{generalized antipodal positions}  in $L$ if  the distance between them is equal to half unit. See figure \ref{gAP}.
\end{defn}

Observe that the skeleton $S=\{(a,b)\in L\times L | a \mbox{ is generalized antipodal to } b\}$ corresponds to the configurations in which the robots are in generalized antipodal position.

\begin{figure}[H]
\centering
\includegraphics[width=0.75\textwidth,keepaspectratio]{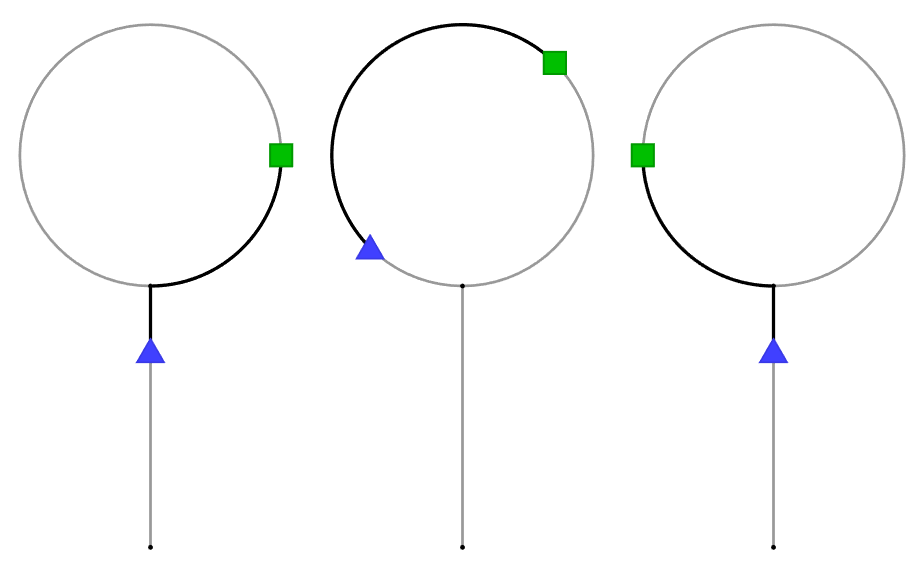}
\caption{Generalized antipodal positions}
\label{gAP}
\end{figure}

Consider the homotopy $H:X \times I \to X$ given by the projection into the skeleton $S$ following the traces of the homotopy as in figure \ref{homotopyII}.

Let ${x_i}$ and $x_f$ be the initial and final states of the robots. We call the \emph{initial generalized antipodal} state $x_i^\prime\:=\:H_1(x_i)$ the image of the initial state in the skeleton $S$ by the homotopy. Likewise, we call the \emph{final generalized antipodal} state $x_f^\prime\:=\:H_1(x_f)$ the projection per homotopy of the final state in  $S$.

As before, the idea of our algorithm will be to move the initial and final states $x_i$ and $x_f$ to their generalized antipodal states $x_i^\prime$ and $x_f^\prime$ respectively, and then apply in the skeleton $S$ a new algorithm to move from $x_i^\prime$ to $x_f^\prime$ that we will explain next.

\subsubsection{Moving the states within the skeleton}
Recall that the skeleton $S$ is the graph shown in figure \ref{chain}. Let $V$ be the set of vertices of the graph $S$. We denote $\tilde V$ the {\em extended set of vertices} given by $\tilde V= V\cup I_1\cup I_2$.

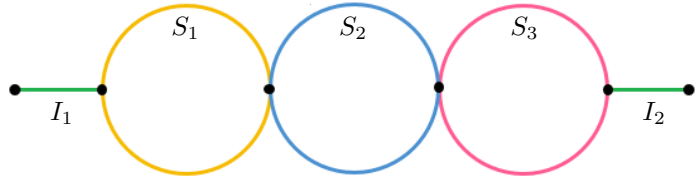
\begin{figure}[H]
\centering
\begin{overpic}
[width=0.75\textwidth,keepaspectratio]{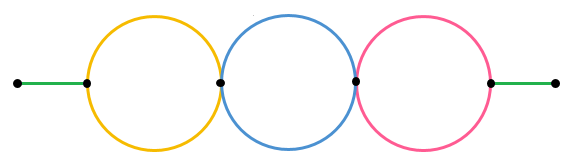}
\put(8,10){$I_1$}
\put(25,22){$S_1$}
\put(48.5,22){$S_2$}
\put(72.5,22){$S_3$}
\put(91,10){$I_2$}
\end{overpic}
\caption{The skeleton $S$}
\label{chain}
\end{figure}

We will decompose the Cartesian product $S\times S$ in three regions. The first region $U_1$ will consist of the pairs of points in $S$, initial and final, such that each is an extended vertex, $U_1=\tilde V\times \tilde V$. The instruction here will be to go shortest path. If traversing a circle, choose always the counterclockwise direction.

The second region is the set of points that are antipodal to each other in any circle but none of them is a vertex, $$U_2= \{(x_i,x_f)\in \hat S_j\times \hat S_j |\; x_i \mbox{ is antipodal to } x_f, \mbox{ for } j=1,2,3\}$$
where $ \hat S_j$ is the circle $S_j$ minus the vertices. The instruction in this region is to go counterclockwise in $S_j$ from the initial to the final state.

The third region is the rest of the Cartesian product, $U_3=(S\times S)\setminus (U_1\cup U_2)$. The instruction in this case is to go following the shortest path and choosing the counterclockwise direction whenever traversing any circle.

\begin{figure}[H]
\centering
\begin{overpic}
[width=0.675\textwidth,keepaspectratio]{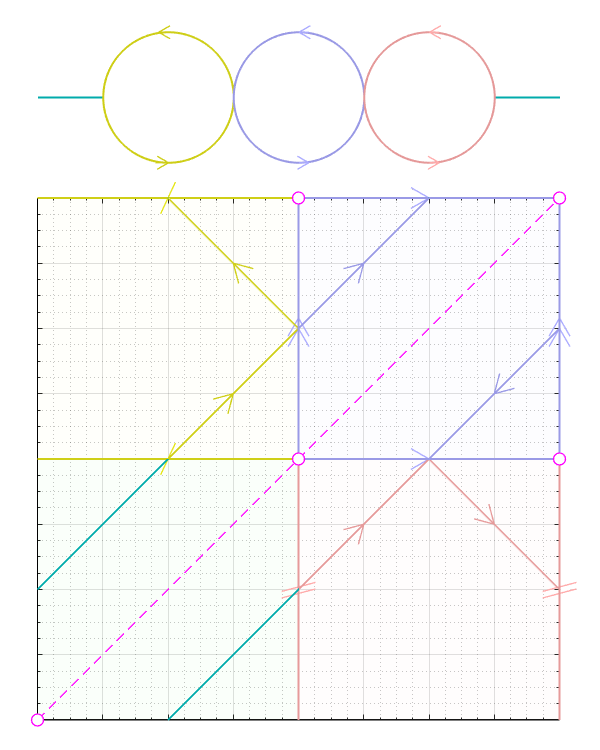}
\put(21,56){$S_1$}
\put(56,56){$S_2$}
\put(56,21){$S_3$}
\put(21,86){$S_1$}
\put(38,86){$S_2$}
\put(56,86){$S_3$}
\end{overpic}
\caption{Counterclockwise directions in each circle $Sj \subset S$}
\label{ccNotation}
\end{figure}
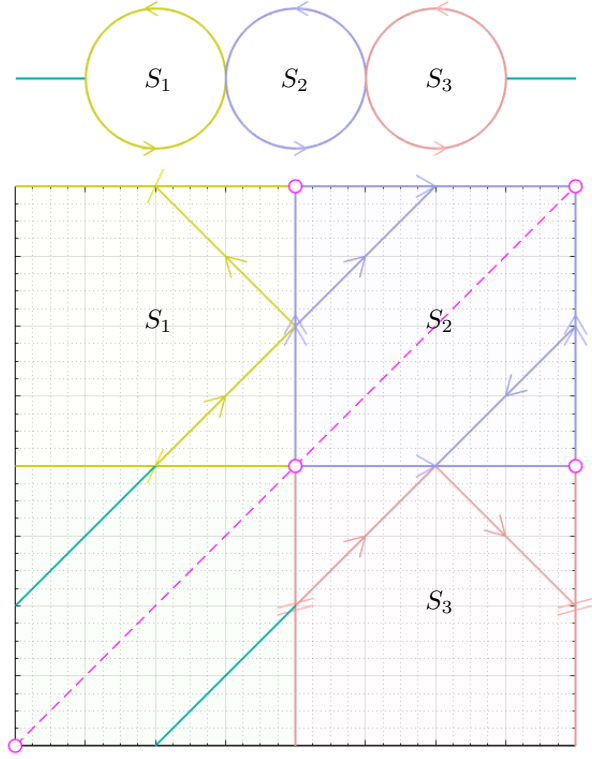	

\subsubsection{Algorithm in the configuration space} 
Let $V_1$ be a small open neighborhood of $U_1$, $V_2$  an open neighborhood of $U_2$ and $V_3=U_3$.
\begin{enumerate}
\item Preliminary step: Move initial state $x_i$  to its generalized antipodal state $x_i^\prime$ on $S$ following the path  $H_t(x_i)$.
\item Main step: While on $S$, if the two generalized antipodal states $x_i^\prime$ and $x_f^\prime$ are in $V_1$, then go to $x_f^\prime$ following shortest path. When crossing any circle $S_j \subset S$ go in counterclockwise direction (see figure \ref{ccNotation}). If $x_i^\prime$ and $x_f^\prime$ are in $V_2$, then move $x_i^\prime$ in counterclockwise direction until it reaches $x_f^\prime$. Otherwise, move $x_i^\prime$ to $x_f^\prime$ following shortest path and in counterclockwise direction whenever crossing any circle $S_j \subset S$.

\item Final step: Move the final antipodal state $x_f^\prime$ back to the final state $x_f$ following the reverse path of $H_t(x_f)$.
\end{enumerate}

\subsubsection{Algorithm in the physical space $L$} 
Our aim is to identify in the physical space the positions corresponding to the different regions described in the configuration space.

First we define certain distinguished positions in the configuration space that will correspond to the generalization of the idea of interchanging robots in the physical space.

\begin{defn}
Two states $x_i=(A_i,B_i)$ and $x_f=(A_f,B_f)$ in $S_j\subset S$ are {\em generalized swapped states} if  
\[ \begin{cases*}
                     A_i=B_f<{1\over 2}  \mbox{ or }
                    A_f=B_i>{1\over 2}
                    & if  $j=1$  \\
                     A_i=B_f \mbox{ and } A_f=B_i & if  $j=2$  \\
                 A_i=B_f>{1\over 2} \mbox{ or }
                     A_f=B_i <{1\over 2} & if  $j=3$
                 \end{cases*} \]

\label{swapedDefinition}
\end{defn}


The generalized swapped states depicted in figure \ref{gssCS} correspond in the physical space to the positions depicted in figure \ref{gss}.

\begin{remark}
Observe that if initial and final states are generalized swapped states, then the initial and final positions of at least one robot are antipodal in the circle of the lollipop $L$.
\end{remark}

\begin{figure}[H]
\centering
\begin{overpic}
[width=0.7\textwidth,keepaspectratio]{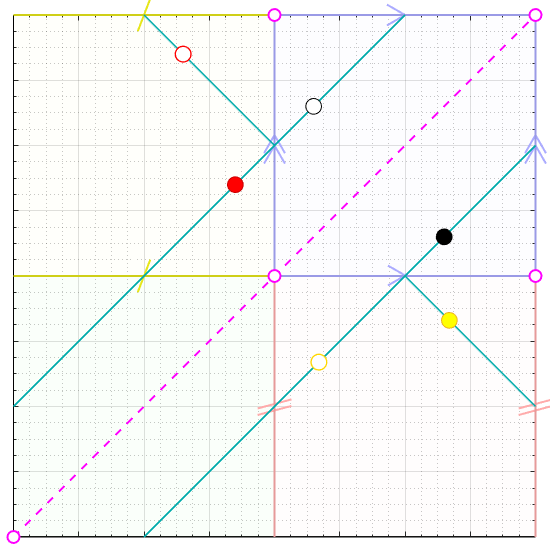}
\put(25,73){$S_1$}
\put(72,73){$S_2$}
\put(72,25){$S_3$}
\end{overpic}
\caption{Generalized swapped states in the configuration space}
\label{gssCS}
\end{figure}	

\begin{figure}[H]
\centering
\begin{overpic}
[width=0.75\textwidth,keepaspectratio]{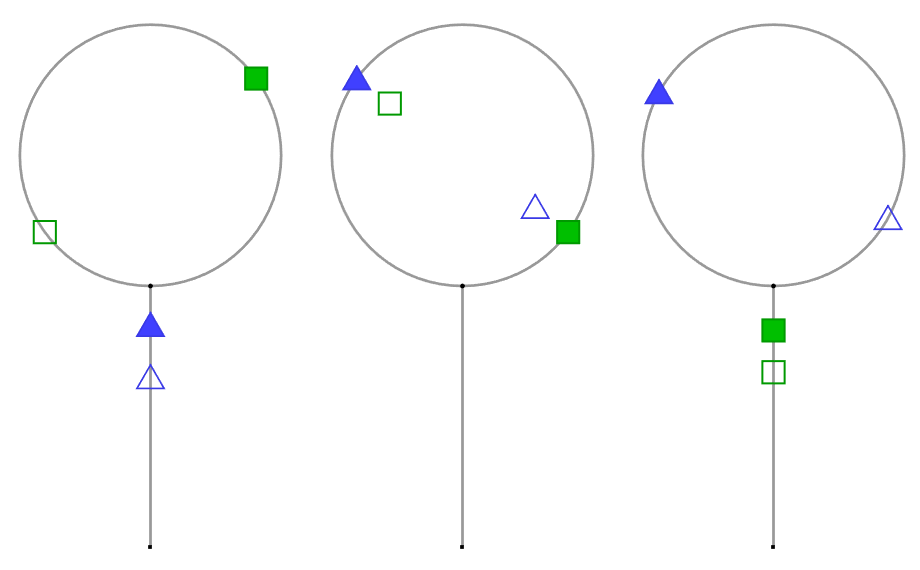}
\put(12,55){in $S_1$}
\put(46,55){in $S_2$}
\put(80,55){in $S_3$}

\end{overpic}
\caption{Generalized swapped states in the physical space}
\label{gss}
\end{figure}	

We define \emph{the extended interval} $\tilde I$ in $L$ as the interval $I$ together with the pole $P$ in $S^1$ as shown in figure \ref{extendedI}, $\tilde I = I\cup \{P\}$.

\begin{defn}
We say that two robots in generalized antipodal position are in {\em generalized vertex position} if they both are in the extended interval $\tilde I$. 
\end{defn}

\begin{figure}[H]
\centering
\begin{overpic}
[width=0.28\textwidth,keepaspectratio]{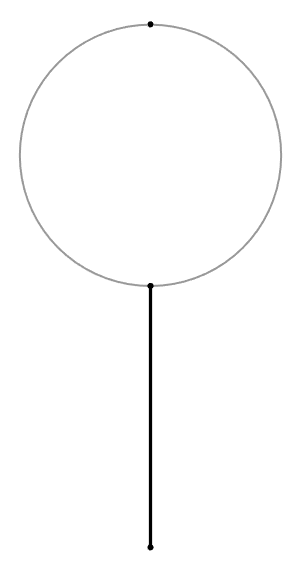}
\put(25,98){$P$}
\end{overpic}
\caption{Extended interval $\tilde I$}
\label{extendedI}
\end{figure}
For robots in generalized vertex position, the initial and final positions of the two robots are either in the same or reverse order.

\begin{defn}
Two pairs of positions $(A_i,B_i)$ and $(A_f,B_f)$ in $\tilde I$ are said to be in the \emph{same order}  if there exists a collision-free translation motion from  $A_i$ to $A_f$ and $B_i$ to $B_f$ within the interval. If no such translation motion exists, then the two pairs of positions are said to be in reverse order.
\end{defn}

\begin{remark}
We observe that initial and final states in generalized vertex position are in the same order if the initial and final positions of the first robot are both in the same half of the interval edge of the lollipop.
\end{remark}

The instructions of the motion planning algorithm in $L$ will be as follows:

\begin{enumerate}
\item Preliminary step: 
 If the distance between the two robots is more than half unit (figure \ref{More}), then move the robot that is closer to the beginning of the interval towards the other robot until both robots are at generalized antipodal position. The other robot remains stationary. Otherwise (figure \ref{Less}), move both robots away from each other (and stay in the circle if moving in the circle) until both robots reach generalized antipodal position.
 \begin{figure}[H]
\centering
\includegraphics[width=0.5\textwidth,keepaspectratio]{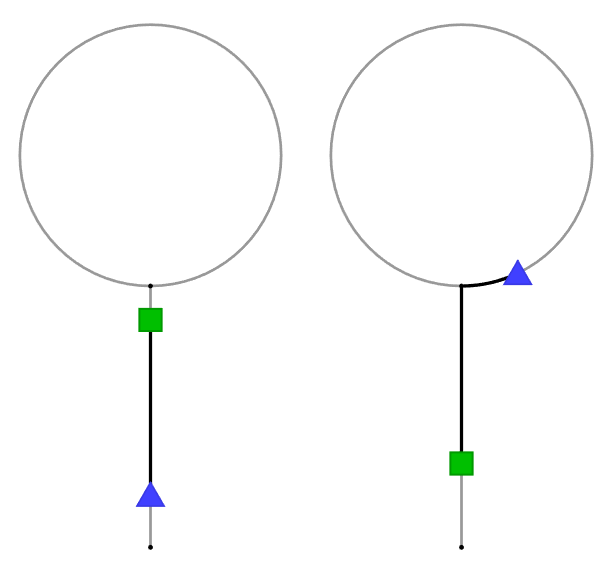}
\caption{Distance greater than half unit.}
\label{More}
\end{figure}

The ratio between the speeds at which the robots are moving is given by the slopes of the traces of the homotopy in figure \ref{homotopyII}. Repeat this procedure to move the final positions to a generalized antipodal position as well.

\begin{figure}[H]
\centering
\includegraphics[width=1\textwidth,keepaspectratio]{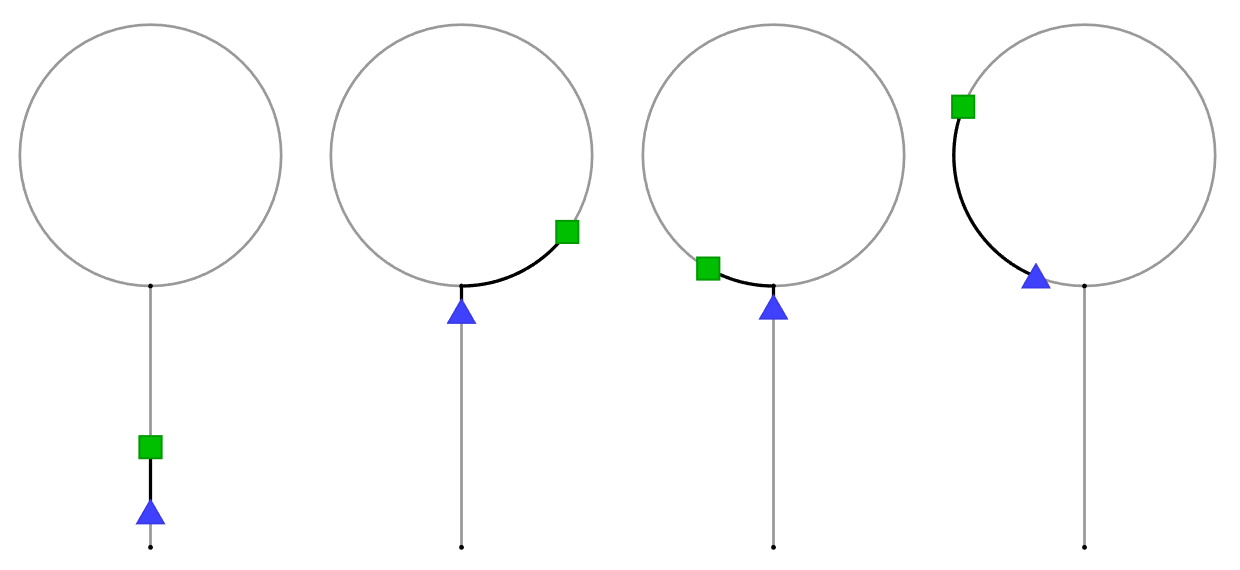}
\caption{Distance shorter than half unit.}
\label{Less}
\end{figure}

\item Main step:
\begin{itemize}
\item If the initial and final generalized antipodal positions are in generalized vertex position, then move the two robots to their final positions following shortest path if they are in the same order, or go around the circle in counterclockwise direction towards the final positions if they are in reverse order.

\item If the initial and final generalized antipodal positions are generalized swapped positions, then move whichever robot is in the circle counterclockwise to its final destination while remaining in the circle. The other robot moves accordingly so that to stay at half unit distance away from the other.

\item Otherwise, for all remaining configurations, move both robots following shortest paths, and counterclockwise whenever in the circle, until they reach their final positions.
\end{itemize}
 
\item Final step: Move both robots to their final positions following shortest paths. Recall that the ratio of the speeds at which the two robots will be moving in this case is given by the slopes of the traces of the homotopy in figure \ref{homotopyII}.

\end{enumerate}
\subsection{Running the algorithm in the physical space $L$}
We will show now some sample cases for each of the domains of continuity.
\begin{enumerate}
\item{Domain of continuity $V_1$: Extended Interval}
\begin{figure}[H]
\centering
\begin{overpic}
[width=1\textwidth,keepaspectratio]{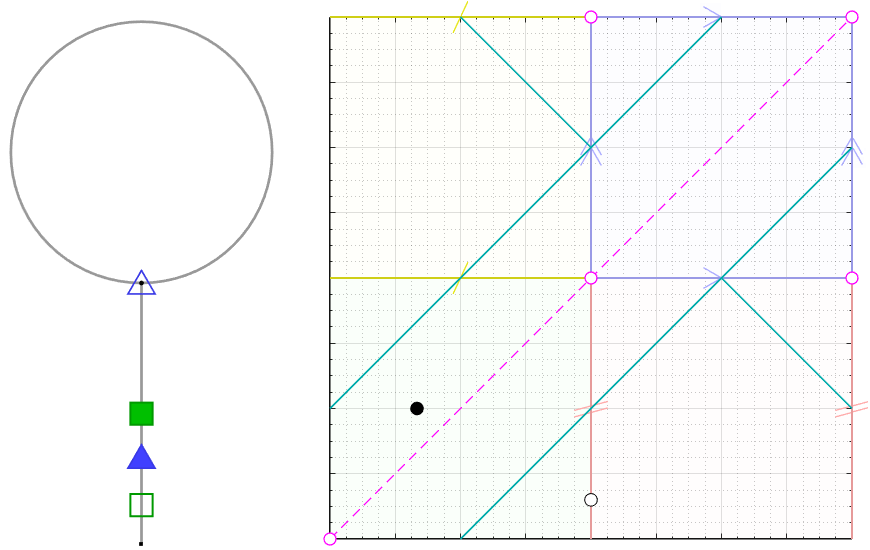}
\end{overpic}
\caption{Initial and final states in $V_1$}
\end{figure}

\begin{figure}[H]
\centering
\includegraphics[width=1\textwidth,keepaspectratio]{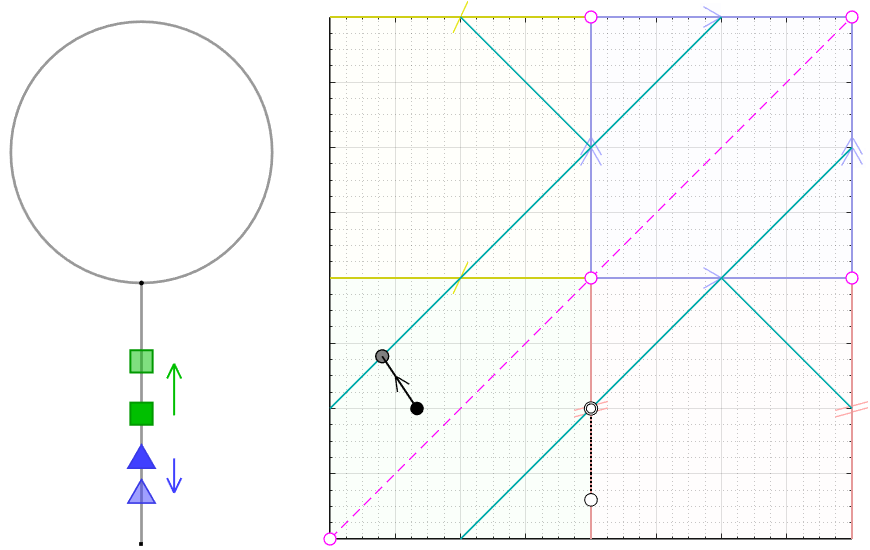}
\caption{Preliminary step in $V_1$}
\end{figure}

\begin{figure}[H]
\begin{adjustwidth}{-1.67cm}{-1.50cm}
\centering
\includegraphics[width=1.3\textwidth,keepaspectratio]{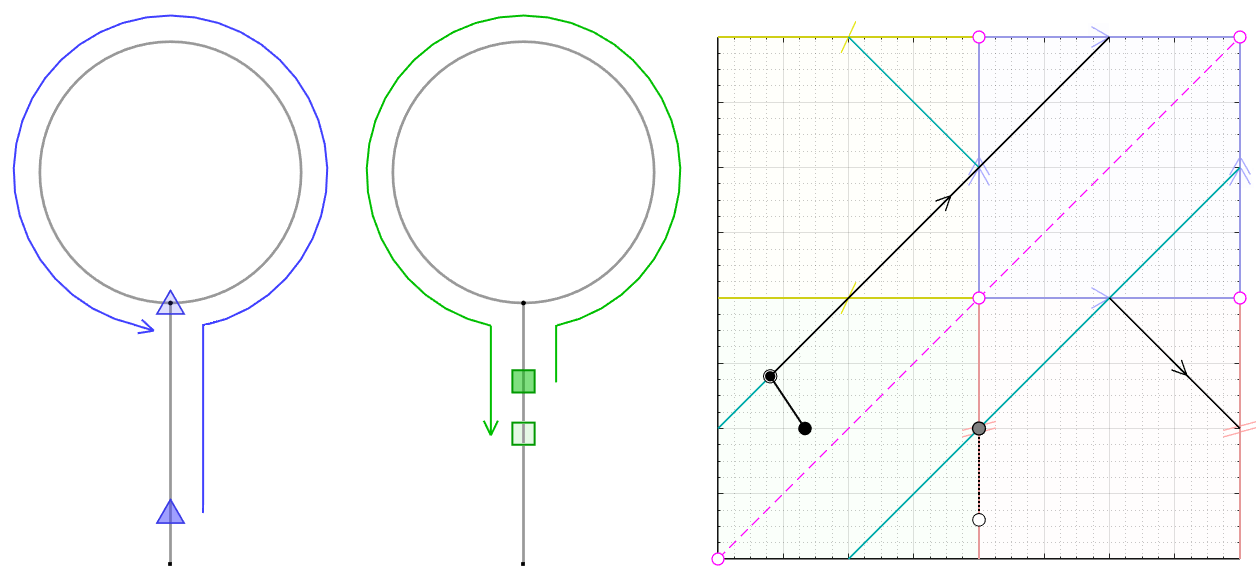}
\caption{Main step in $V_1$}
\end{adjustwidth}
\end{figure}

\begin{figure}[H]
\centering
\includegraphics[width=1\textwidth,keepaspectratio]{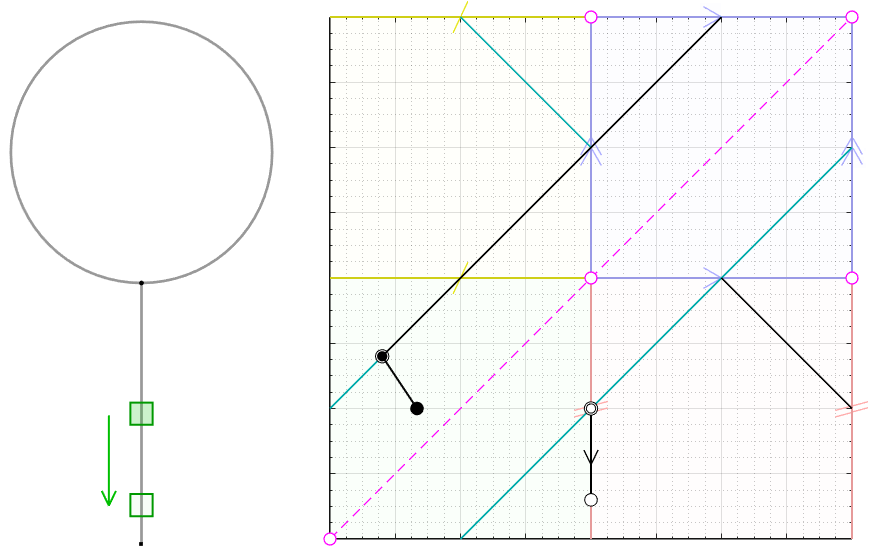}
\caption{Final step in $V_1$}
\end{figure}

\item{Domain of continuity $V_2$: Generalized swapped positions.}

\begin{figure}[H]
\centering
\includegraphics[width=1\textwidth,keepaspectratio]{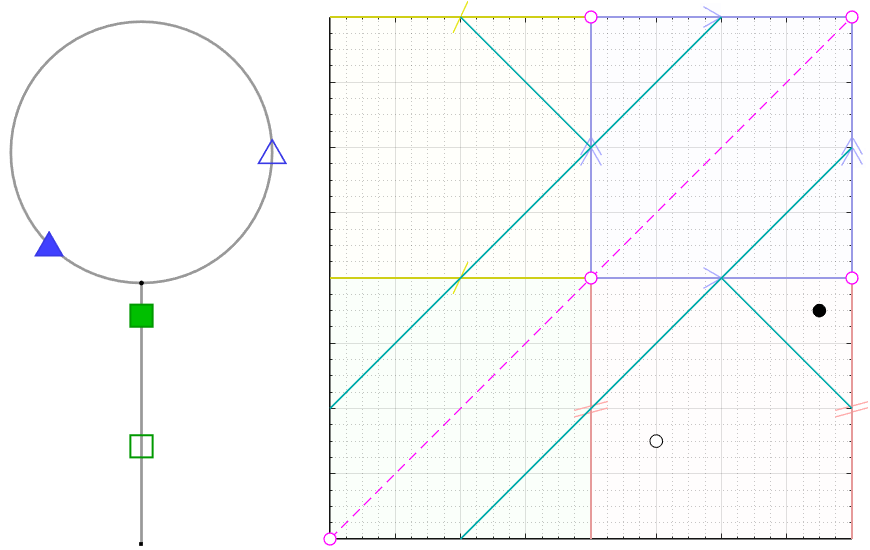}
\caption{Initial and final positions in  $V_2$}
\end{figure}

\begin{figure}[H]
\centering
\includegraphics[width=1\textwidth,keepaspectratio]{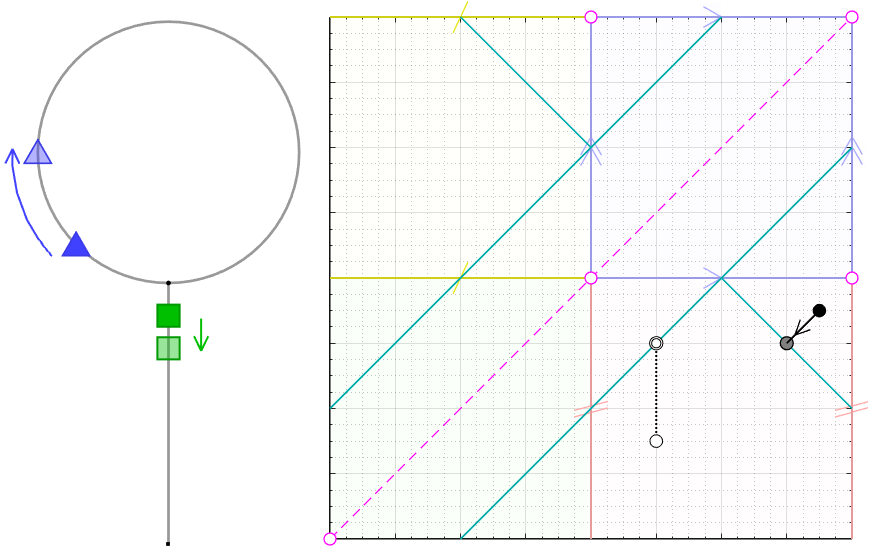}
\caption{Preliminary step in $V_2$}
\end{figure}

\begin{figure}[H]
\centering
\includegraphics[width=1\textwidth,keepaspectratio]{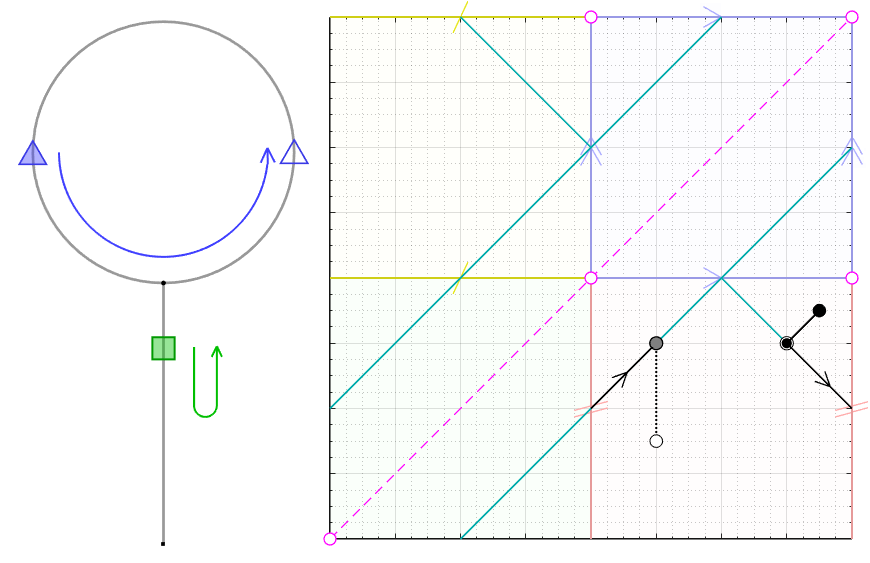}
\caption{Main step in $V_2$}
\end{figure}

\begin{figure}[H]
\centering
\includegraphics[width=1\textwidth,keepaspectratio]{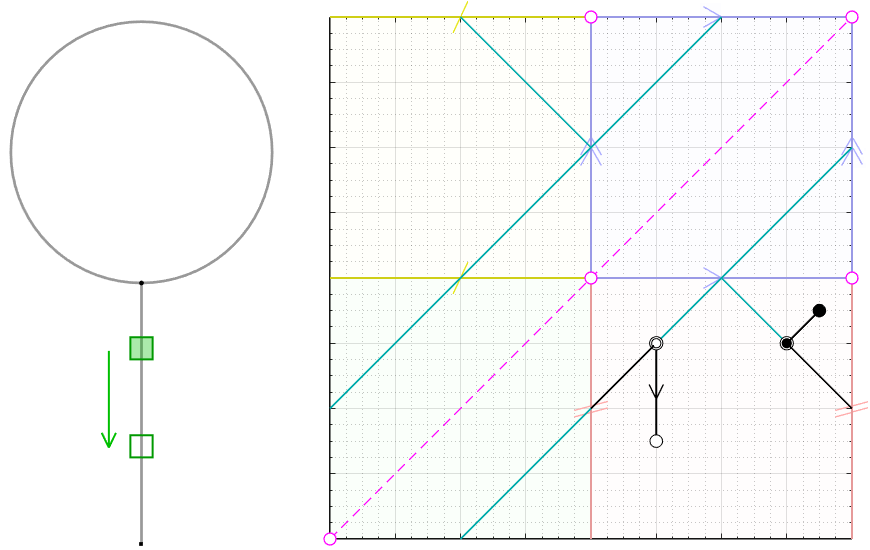}
\caption{Final step in $V_2$}
\end{figure}

\item{Domain of continuity $V_3$}

\begin{figure}[H]
\centering
\includegraphics[width=1\textwidth,keepaspectratio]{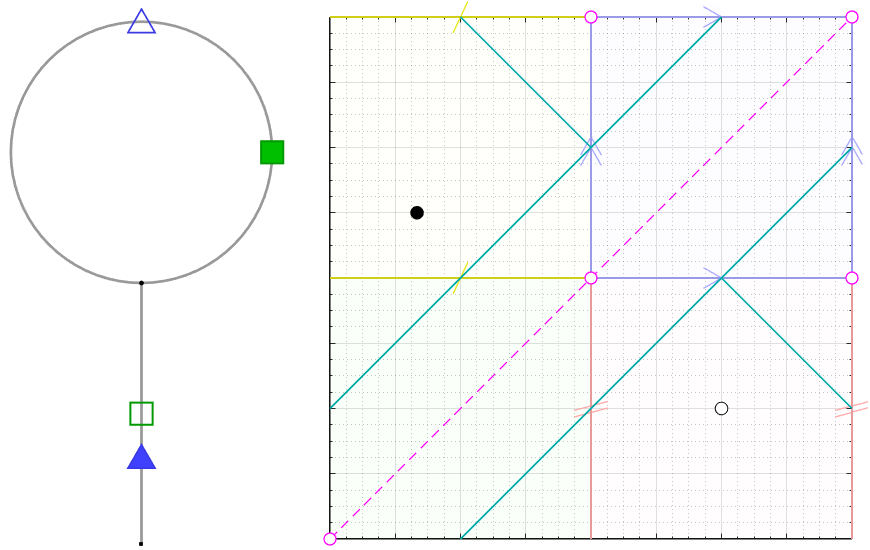}
\caption{Initial and final states in $V_3$}
\end{figure}

\begin{figure}[H]
\centering
\includegraphics[width=1\textwidth,keepaspectratio]{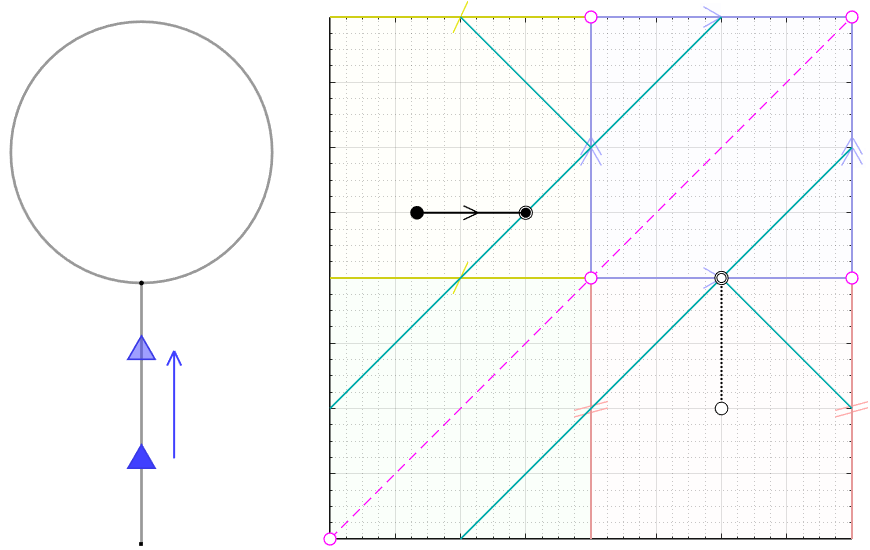}
\caption{Preliminary step in $V_3$}
\end{figure}

\begin{figure}[H]
\begin{adjustwidth}{-2.00cm}{-1.50cm}
\centering
\includegraphics[width=1.35\textwidth,keepaspectratio]{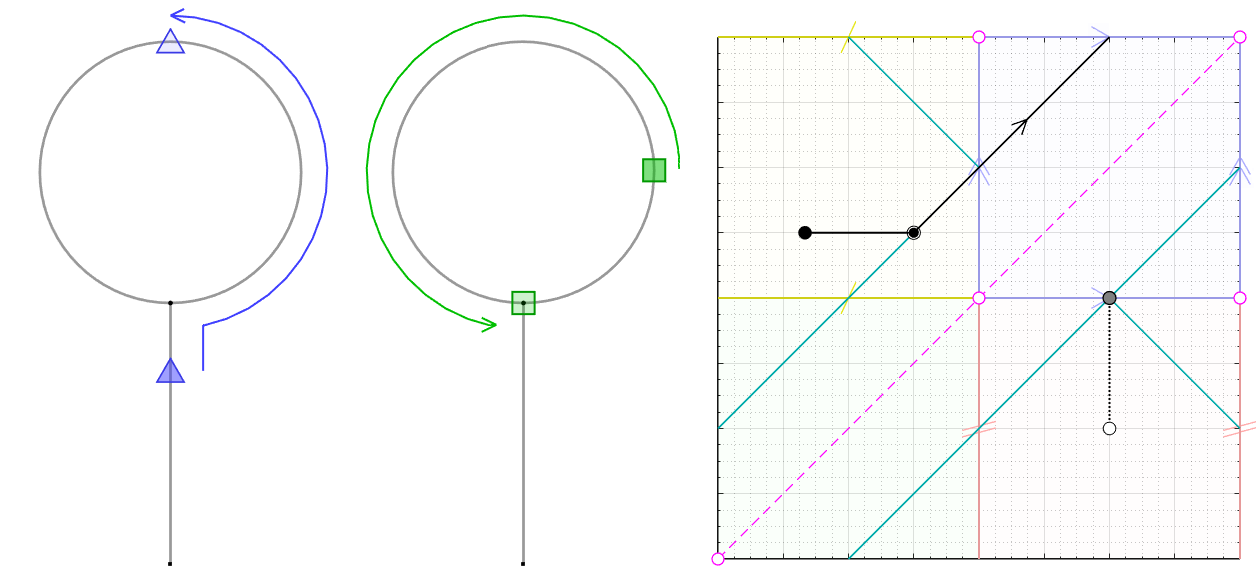}
\caption{Main step in $V_3$}
\end{adjustwidth}
\end{figure}

\begin{figure}[H]
\centering
\includegraphics[width=1\textwidth,keepaspectratio]{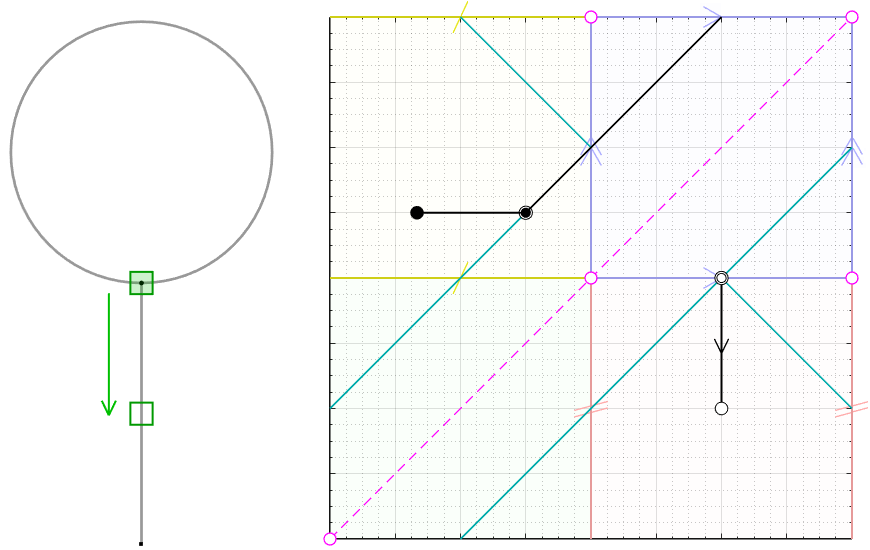}
\caption{Final step in $V_3$}
\end{figure}
\end{enumerate}

\section{Algorithm implementation}
A demonstration video for both algorithms implemented in MATLAB is available online at \url{https://figshare.com/articles/Demonstration_video_for_collision-free_motion_planning_algorithms_in_topological_robotics/8019383} (or via the shortened web link: \url{https://bit.ly/demoVideo_MPAs})
\bibliography{AB}
\bibliographystyle{plain}
\end{document}